\documentclass{article}
\usepackage[utf8]{inputenc}
\usepackage{amssymb, amsfonts,amsthm}
\usepackage{amsmath, mathpazo}
\usepackage{mathtools}
\usepackage{commath}
\usepackage{multicol}
\usepackage{graphicx,subfigure}
\usepackage{color}
\usepackage{titlesec} 
\usepackage[margin=1in]{geometry}
\numberwithin{equation}{section}

\theoremstyle{plain}
\newtheorem{theorem}{Theorem}[section]
\newtheorem{lemma}{Lemma}[section]
\newtheorem{proposition}{Proposition}[section]

\theoremstyle{definition}
\newtheorem{definition}{Definition}[section]

\theoremstyle{remark}
\newtheorem{remark}{Remark}[section]

\newcommand{\helen}[1]{{\color{black}#1}}
\newcommand{\elaine}[1]{{\color{black}#1}}
\def\<{\langle}
\def\>{\rangle}

\def\Z{\mathbb{Z}}

\title{A One-Dimensional Symmetric Force-Based Blending Method for Atomistic-to-Continuum Coupling}
\author{
{Elaine Gorom-Alexander\footnote{Department of Mathematics and Statistics, University of North Carolina at Charlotte, Email: egorom@uncc.edu }} \and {Xingjie Helen Li\footnote{Department of Mathematics and Statistics, University of North Carolina at Charlotte, Email: xli47@uncc.edu}}
}

\begin{document}

\maketitle

\begin{abstract} Inspired by the blending method developed by [P. Seleson, S. Beneddine, and S. Prudhome, \emph{A Force-Based Coupling Scheme for Peridynamics and Classical Elasticity}, (2013)] for the nonlocal-to-local coupling, we create a symmetric and consistent blended force-based Atomistic-to-Continuum (a/c) scheme for the atomistic chain in one-dimensional space.
The conditions for the well-posedness of the underlying model are established by analyzing an optimal blending size and blending type to ensure the $H^1$ semi-norm stability for the blended force-based operator. We present several numerical experiments to test and confirm the theoretical findings.
\end{abstract}

\textbf{Keywords:}
 atomistic-to-continuum coupling, symmetric force-based blending, stability and blending size analysis
\section{Introduction}

Many important materials from airplane wings to computer chips can be improved by a better
understanding of failure modes such as fracture and fatigue. Thus, one of the most important
goals of computational materials science is to efficiently and reliably predict phenomena such as
crack growth and to facilitate the design of new materials better able to resist failure. Scientists and engineers have proposed several multi-scale methods (e.g., \cite{phillips2001,tadmor2011,WeinanEbook2011,Hou2012,Du2012a,Kondov2013,Silling2000,seleson2013aa,Seleson2015a,Ortizbook2016}) to overcome the computational challenges in fidelity and efficiency.

The two main strategies in multi-scale modelings are: (1) bottom-up atomistic-to-continuum:
coarse-graining of microscopic descriptions (e.g., atomistic models) of material behavior \cite{Tadmor2009a,tadmor2011,WeinanEbook2011,Shapeev2011,Luskin2013a}; (2) top-down local-to-nonlocal: informing macroscopic models (e.g., continuum equations) with physics gleaned from the microscopic scales \cite{Silling2000,Du2012a,seleson2013aa,Seleson2015a,NewsDu,Delia2014a,Delia2015a,Ortizbook2016,d2019review,YueYu2019,YueYu2020}.
The former provides a "closer" comparison with macroscopic experiments, and the latter predicts the materials' microscopic properties \cite{Kondov2013}. Meanwhile, the two approaches have many interconnections, and their development and understanding often inspire each other.

In this work, we employ a symmetric blending strategy developed by P Seleson {\it et. al.} for the  nonlocal-to-local coupling \cite{seleson2013aa,Seleson2015a} and develop a new force-based atomistic-to-continuum model for a 1D atomistic chain.
We then study the stability property of the new coupling scheme in terms of the blending function and its blending size using similar mathematical tools from \cite{li2012}. We investigate the optimal number of atoms  within the blending region to ensure the positive-definiteness of the resulting force blending operator under the discrete $H^{1}$ semi-norm. The results admit a very narrow blending region to maintain the coercivity and efficiency when the number of atoms is large. In addition, the stability analysis developed in this work is crucial for the convergence for several popular iterative methods for solving large-force equilibrium systems.

We will arrange the paper as follows. In section 2, we introduce the force-based symmetric blending method for a 1D atomistic chain. We construct an atomistic, linearized force equation and a continuum, linearized force equation from the atomistic energy equation. The consistency between these methods is discussed in Proposition \ref{proposition:forceconsistency}. A blending function is then introduced to symmetrically combine these two force-based equations.

In section 3, we establish the optimal conditions on the size of the blending region for the blending force operator with respect to the $H^1$ stability. Theorem \ref{theorem:blendingregion} and Theorem \ref{corollary: blendsize} establish these conditions.

In section 4, a uniform stretch is applied to compute the critical strain errors for various types of 
blending functions with different blending sizes. It is found that the cubic blending function is optimal.

We find that a larger polynomial size on the blending region is suggested from the numerical trials when the number of atoms in the chain is just moderately large.

Also in section 4, we test a sine and Gaussian external force to the system to model the displacement with our force-based blending method. 
The displacements produced by the blending methods with sufficient blending size agree with those of fully atomistic models. In addition, we compare the impact of interaction range of the atomistic model, and observe that an interaction range potential greater than $2$ neighbors does not change the displacement significantly.

\section{Derivation of the Symmetric and Consistent Force-Based Scheme}

In this section, we will introduce notations, introduce the reference atomistic model, and then derive the continuum approximation.
After this, we introduce a blending equation to symmetrically combine these two models. Utilizing this blending equation, a coupling scheme for the blended atomistic and continuum forces is created.

\subsection{Notations}
We consider a 1D atomistic chain with finite interaction range up to the $N$-th nearest neighbor and a total number of $2M$ atoms within the domain $\Omega$. We denote the scaled reference lattice $x_{\ell}=a\ell$ for $\ell \in \Z$ with fixed reference  lattice spacing constant $a:=\frac{1}{M}$ such that \helen{we can select a reference domain which is fixed to be $\Omega =(-1, 1]$.} Throughout, \helen{the interaction range $N$ will be fixed.} 
The chain is deformed to a current configuration $y_{\ell}=x_{\ell}+u_{\ell}$.

The displacement field $u=(u_{\ell})_{\ell \in \Z}:\mathbb{Z}\rightarrow \mathbb{R}$ is assumed to be $2M$ periodic discrete function and  $\mathcal{U}$ denotes the space of all $2M$ periodic displacement functions
\begin{equation*}
    \mathcal{U}:=\{u:\,u_{\ell+2M}=u_{\ell}|\ell \in \Z \}.
\end{equation*}
Accordingly, we set the deformation space by
\begin{equation*}
    \mathcal{Y}:=\{y:\, y_{\ell}=x_{\ell}+u_{\ell} |u \in \mathcal{U}, \ell \in \Z \}.
\end{equation*}
We also define the discrete differentiation operator for simplicity, $u^{'}$, on periodic displacements by
$$u^{'}_{\ell}:=\frac{u_{\ell+1}-u_{\ell}}{a}.$$
Then we may define the higher-order discrete differentiation $u^{''}$, $u^{(3)}$, and $u^{(4)}$ for $\ell$ by
\begin{align}\label{def_FiniteDiff}
\begin{cases}
u^{''}_{\ell}:=\frac{u'_{\ell}-u'_{\ell-1}}{a},\\
u^{(3)}_{\ell}:=\frac{u''_{\ell+1}-u''_{\ell}}{a},\\
 u^{(4)}_{\ell}:=\frac{u^{(3)}_{\ell}-u^{(3)}_{\ell-1}}{a}.
\end{cases}
\end{align}
\noindent
For a displacement $u\in \mathcal{U}$ and its discrete derivatives, we employ the discrete $\ell^{2}$ and $\ell^{\infty}$ norms by
\begin{align}\label{def_norm}
\norm{u}_{\ell^{\elaine{2}}}^{2}:= \sum_{\ell=-M+1}^{M}\abs{u_\ell}^{2}a,\quad
\text{ and }\quad \norm{u}_{\ell^{\infty}}:=\max_{-M+1\le \ell \le M} \abs{u_{\ell}}.
\end{align}
In particular, the associated inner product for $\ell_{2}$ is $$\<u,w\>:= \sum_{\ell=-M+1}^{M}u_{\ell}w_{\ell}\,a.$$ We also employ the discrete $H^{1}$ semi norm, $\abs{u}_{H^{1}}^{2}=\norm{u^{'}}_{\ell_{2}}^{2}$ in the stability analysis.

\elaine{Meanwhile, we proceed with $\widetilde{u}:\mathbb{R}\rightarrow \mathbb{R}$ as a quintic spline interpolation of $u$ such that
\begin{equation}\label{quintic_spline}
\begin{split}
\widetilde{u}(a\ell)&=u_{\ell},\\
\widetilde{u}(-Ma)& =\widetilde{u}(Ma),\\
\lim_{t\rightarrow (a\ell)^{-}}\frac{d^{\omega}\widetilde{u}}{dx^{\omega}}(t)
&=\lim_{t\rightarrow (a\ell)^{+}}\frac{d^{\omega}\widetilde{u} }{d x^{\omega}}(t),\quad \omega = 1,\dots, 4
\end{split}
\end{equation}
As $\widetilde{u}$ is a continuous function, we can introduce notations for its derivatives, for instance, $\widetilde{u}_x$ as its first derivative at $(a\ell)$, and  $\widetilde{u}_{xx}$ as its second derivative at $(a\ell)$, etc.}

We can compare the derivatives of $\widetilde{u}(x)$ with the differencing of $u_\ell$. Clearly, we have
\begin{equation}\label{deriv_vs_diff}
\begin{split}
    u'_{\ell} =& \widetilde{u}_{x}(a\ell)+ \frac{a}{2} \widetilde{u}_{xx}(\xi)\\
    u^{''}_{\ell} =& \widetilde{u}_{xx}(a\ell)+ \frac{a^2}{12} \widetilde{u}_{xxxx}(\tilde{\xi})
\end{split}
\end{equation}
 \elaine{Note that throughout, subscript $\ell$ is used to denote when discrete differentiation is employed; whereas subscript $x$ is used to denote when considering the true derivatives.} 

As in \cite{li2012}, we will frequently use the following discrete summation by parts identity:
\begin{lemma}\label{lemma:summation}
\noindent Suppose $\{u_{\ell}\}^{b}_{\ell=a+1}$ and $\{v_{\ell}\}^{b}_{\ell=a+1}$ are two sequences, then we have
\[\sum_{\ell=a+1}^{b}u_{\ell}(v_{\ell}-v_{\ell-1})=[u_{b}v_{b}-u_{a}v_{a}]-\sum_{\ell=a+1}^{b}(u_{\ell}-u_{\ell-1})v_{\ell-1}.
\]
Furthermore, when both  $\{u_{\ell}\}^{b}_{\ell=a+1}$ and $\{v_{\ell}\}^{b}_{\ell=a+1}$ are periodic sequences with
$u(a)=u(b)$ and $v(a)=v(b)$, we have
\[
\sum_{\ell=a+1}^{b}u_{\ell}(v_{\ell}-v_{\ell-1})=-\sum_{\ell=a+1}^{b}(u_{\ell}-u_{\ell-1})v_{\ell-1}.
\]
\end{lemma}
We use this lemma to find conditions on coupling to ensure the positive-definiteness of the bilinear form of the symmetric, blended force-based operator.

 \subsection{1D atomisitic and continuum models}
{\bf Energy formulations and consistency analysis.\\}
We now consider a one-dimensional  atomistic periodic chain deformed into configuration $y\in \mathcal{Y}$. Recall that the atomistic periodicity is fixed to $2M$, and the interaction range is fixed to $N$-th neighbors. The total atomistic energy  for this periodic chain is given by
\begin{equation}\label{def_AtomicE}
E^{a,tot}(y):=\sum_{\ell=-M+1 }^{M} \sum_{\substack{k=-N, \\ k\neq0}}^{N} \frac{a}{2}  \phi\big(\frac{y_{\ell+k}-y_{\ell}}{a}\big),
\end{equation}
where $\phi(\cdot):\mathbb{R}\rightarrow \mathbb{R}$ is a Lennard-Jones type potential. It can also be viewed as the energy density per unit volume for pairwise interactions. We assume that $\phi(\cdot)$ has the following properties:
\begin{itemize}
    \item $\phi(r)=\phi(\abs{r})$;
    \item $\phi$ is at least four times differentiable; and
    \item $\phi_{xx}(1)>0$ and $\phi_{xx}(k) \leq 0$ for $k \geq 2$.
\end{itemize}
In the numerical simulation, we employ the Morse potential, and a graphical illustration can be found in Figure~\ref{fig:1D_morsePot}.

\begin{figure}[htp!]
    \centering
    \includegraphics[width=0.5\textwidth, height=6.0 cm]{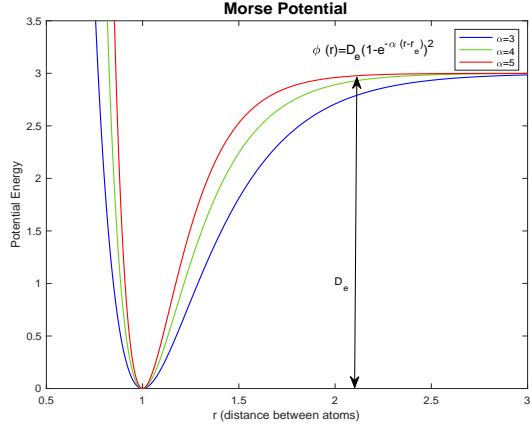}
    \caption{Graphical illustration of the Lennard-Jones type interaction potential (Morse potential) used in numerical experiments. Notice that the local minimum is achieved at the nearest neighbour distance $r=r_e$. 
    }
    \label{fig:1D_morsePot}
\end{figure}

Next, we derive the continuum model by only using the first neighbor distance, that is we \elaine{only use the differencing $u_{\ell}^{'}$ as we approximate the argument for $\phi$.}
\elaine{For $k=2,\dots, N$, we have
\begin{equation*}
    \begin{split}
        \frac{y_{\ell+k}-y_{\ell}}{a}  &=  \frac{x_{\ell+k}-x_{\ell}+u_{\ell+k}-u_{\ell}}{a}\\
        &\elaine{=} k+u'_{\ell}+
        \sum_{j=1}^{k-1}u'_{\ell+j},
    \end{split}
\end{equation*}
Using the error estimates listed in \eqref{deriv_vs_diff}, we can replace $u'_{\ell+j}$ by $u'_{\ell}$ and estimate the discrepancy
\begin{equation}\label{def_YApprox}
    \begin{split}
        \frac{y_{\ell+k}-y_{\ell}}{a}  &=k+u'_{\ell}+\sum_{j=1}^{k-1}\widetilde{u}_{x}\left(a(\ell+j)\right)+c_1 a\\
        &=k+u'_{\ell}+\sum_{j=1}^{k-1}\widetilde{u}_{x}\left((a\ell)\right)+c_2a\\
        &=k+ku'_{\ell}+c_3a,
    \end{split}
\end{equation}
where $c_1$, $c_2$ and $c_3$ are constants depending on $k$ and regularity of $\widetilde{u}$.
For $k=-N,\dots, -2$, we can obtain similar consistency estimates. }

The atomistic energy equation is rooted in discrete, nonlocal energy descriptions. 
 Assuming the finest mesh with each atom regarded as a node and substituting the previous approximation into the atomistic energy \eqref{def_AtomicE}, we thus defined the continuum energy as
\begin{equation}\label{def_ContinuumE}
 E^{c}(u):=\sum_{\ell=-M+1 }^{M} \sum_{\substack{k=-N, \\ k\neq0}}^{N} \frac{a}{2}  \phi(k+ku_{\ell}^{'}) .
 \end{equation}
 So far, both the atomistic energy and the continuum energy are non-linearly dependent on the displacement field $\{u_\ell\}_{\ell=-M+1}^{M}$, and we would like to apply further simplifications to obtain linear models.
 \elaine{
 To linearize the total atomistic energy, we follow a similar argument as deriving the continuum energy,
 \begin{equation*}
 \begin{split}
     \frac{y_{\ell+k}-y_{\ell}}{a}&=\frac{x_{\ell+k}-x_{\ell}+u_{\ell+k}-u_{\ell}}{a} \\
     &=k+\frac{u_{\ell+k}-u_{\ell}}{a}.
 \end{split}
 \end{equation*}
 Therefore, the total atomistic energy can be written as
 \begin{equation}\label{def_AtomicE2}
     \begin{split}
         E^{a,tot}(y)&:=\sum_{\ell=-M+1 }^{M} \sum_{\substack{k=-N, \\ k\neq0}}^{N} \frac{a}{2}  \phi\big(\frac{y_{\ell+k}-y_{\ell}}{a}\big) \\
         &=\sum_{\ell=-M+1 }^{M} \sum_{\substack{k=-N, \\ k\neq0}}^{N} \frac{a}{2} \phi\big(k+\frac{u_{\ell+k}-u_{\ell}}{a}\big)
     \end{split}
\end{equation}

Next, we utilize Taylor expansion to $\phi\big(k+\frac{u_{\ell+k}-u_{\ell}}{a}\big)$ at the reference configuration to linearize the expression.
\[ \phi(k+\frac{u_{\ell+k}-u_{\ell}}{a})= \phi(k)+\frac{u_{\ell+k}-u_{\ell}}{a}\phi_{x}(k)+\frac{1}{2}(\frac{u_{\ell+k}-u_{\ell}}{a})^{2}\phi_{xx}(k)+O\left((\frac{u_{\ell+k}-u_{\ell}}{a})^{3}\right).  \]

\noindent Inserting the Taylor approximation into the atomistic energy equation \eqref{def_AtomicE2}, we obtain
\begin{equation}\label{def_linearatomE}
    \sum_{\ell=-M+1 }^{M} \sum_{\substack{k=-N, \\ k\neq0}}^{N} \frac{a}{2}\left[\phi(k)+\frac{u_{\ell+k}-u_{\ell}}{a}\phi_{x}(k)+\frac{1}{2}(\frac{u_{\ell+k}-u_{\ell}}{a})^{2}\phi_{xx}(k)+O\left((\frac{u_{\ell+k}-u_{\ell}}{a})^{3}\right)\right].
\end{equation}
Then, without loss of generality, we assume \elaine{$\sum\limits_{\substack{k=-N, \\ k\neq0}}^{N}\phi(k)=0$ }since this term will not contribute to force. Also, since the reference configuration is a local minimizer and the potential is symmetric, we have
\[\sum_{\substack{k=-N, \\ k\neq0}}^{N} \frac{u_{\ell+k}-u_{\ell}}{a}\phi_{x}(k)=0.\]

Thus, the linearized atomistic energy is
\begin{equation}\label{LinearAtomE}
    \begin{split}
        E^{a,lin}(u)&:= \sum_{\ell=-M+1 }^{M} \sum_{\substack{k=-N, \\ k\neq0}}^{N}\frac{a}{2}\left(\frac{1}{2}(\frac{u_{\ell+k}-u_{\ell}}{a})^{2}\phi_{xx}(k)\right).
    \end{split}
\end{equation}

Following the linearization of the atomistic energy, for the continuum energy, we utilize Taylor expansion to $\phi(k+ku_{\ell}^{'})$ at the reference configuration to linearize the expression
\[ \phi(k+ku_{\ell}^{'})\elaine{=} \phi(k)+ku_{\ell}^{'}\phi_{x}(k)+\frac{1}{2}(ku_{\ell}^{'})^{2}\phi_{xx}(k)\elaine{+O\left((ku'_{\ell})^{3}\right)}.  \]

\noindent Inserting the Taylor approximation into the continuum energy equation \eqref{def_ContinuumE}, we obtain
\begin{equation}\label{def_ContinuumE2}
\begin{split}
&\sum_{\ell=-M+1}^{M} \sum_{\substack{k=-N, \\ k\neq0}}^{N} \frac{a}{2}\left(\phi(k)+ku_{\ell}^{'}\phi_{x}(k)+\frac{1}{2}\left(ku_{\ell}^{'}\right)^{2}\phi_{xx}(k)\elaine{+O\left((ku'_{\ell})^{3}\right)}\right).
\end{split}
\end{equation}
\noindent Following the same justification as above regarding the terms of the Taylor expansion, the linearized continuum energy associated with the finest mesh is defined as
\begin{equation}
\label{def_ContLinE}
\begin{split}
E^{c,lin}(u) & := \sum_{\ell=-M+1}^{M} \sum_{\substack{k=-N ,\\k\neq0}}^{N}\frac{a}{2}\left(\frac{k^{2}}{2}(u_{\ell}^{'})^{2}\phi_{xx}(k)\right) \\
& =\sum_{\ell=-M+1}^{M} \left(\sum_{k=1}^{N} \frac{k^{2}}{2}\phi_{xx}(k)\right)(u_{\ell}^{'})^{2}a.
\end{split}
\end{equation}}
\elaine{
\begin{remark}
Notice that the classical continuum mechanics for interaction range up to the $N$-th nearest neighbour has the following form:
\begin{equation} \label{def_LinearizedCE}
\widetilde{E}^{c,lin}(\widetilde{u})=\int W\left(\frac{d\widetilde{u}}{dx}\right)^{2}dx
\end{equation}
\noindent where $\frac{d\widetilde{u}}{dx}$ is the deformation field and $W$ represents the strain energy density
\begin{equation} \label{def_ContinuumEDensity}
W:= \sum_{k=1}^{N} \frac{k^{2}}{2}\phi_{xx}(k) .
\end{equation}
Notice that if we use a fine mesh and applying a Riemann sum to approximate the integral of \eqref{def_LinearizedCE}, we can convert \eqref{def_LinearizedCE} into \eqref{def_ContLinE}.
Therefore, the linearized continuum energy for a discrete lattice system is consistent with the theory of classical continuum mechanics.
\helen{For the consistency between the classical continuum mechanics and the linearized continuum energy, we refer to Section \ref{appendix} as it follows closely with the consistency for the linearized continuum energy equation and the atomistic energy description.
}
\end{remark}}

Since we obtain the atomistic and continuum energies, in the next proposition, we will summarize the truncation errors.
\begin{proposition}[Consistency Analysis of Linearized Energy Formulations]
Given a fully refined continuum mesh on the 1D atomistic chain, we derive the linearized continuum energy equation \eqref{def_ContLinE},
\elaine{\begin{equation*}
E^{c,lin}(u):  =\sum_{\ell=-M+1}^{M} \sum_{\substack{k=-N ,\\k\neq0}}^{N}\frac{a}{2}\left(\frac{k^{2}}{2}(u_{\ell}^{'})^{2}\phi_{xx}(k)\right).
\end{equation*}
from the atomistic energy description, \eqref{def_AtomicE},
\begin{equation*}
    E^{a,lin}(u):=\sum_{\ell=-M+1 }^{M} \sum_{\substack{k=-N, \\ k\neq0}}^{N}\frac{a}{2}\left(\frac{1}{2}(\frac{u_{\ell+k}-u_{\ell}}{a})^{2}\phi_{xx}(k)\right).
\end{equation*}
with the deformed configuration and the displacement field linked by $y_{\ell}=x_{\ell}+u_\ell$.
Then, the consistency between the linearized continuum energy equation and the atomistic energy equation is $O(a^2)$.}
\end{proposition}

\begin{proof}
\elaine{Comparing $E^{c,lin}$ and $E^{a,lin}$, we have
\begin{equation*}
    E^{c,lin}-E^{a,lin}=\sum_{\ell=-M+1}^{M} \sum_{\substack{k=-N ,\\k\neq0}}^{N}\frac{a}{2}\left(\frac{k^{2}}{2}(u_{\ell}^{'})^{2}\phi_{xx}(k)\right)-\sum_{\ell=-M+1 }^{M} \sum_{\substack{k=-N, \\ k\neq0}}^{N}\frac{a}{2}\left(\frac{1}{2}(\frac{u_{\ell+k}-u_{\ell}}{a})^{2}\phi_{xx}(k)\right).
\end{equation*}

For any $k=2,...,N$, we compare $u_{\ell+k}$ around $\ell$. Recall the quintic spline interpolation $\widetilde{u}$ defined in \eqref{quintic_spline}, we have
\begin{equation*}
    u_{\ell+k}=u_{\ell}+ka \widetilde{u}_{x}+O(a^2).
\end{equation*}

Thus, for all $k$,
\begin{equation*}
    u_{\ell+k}-u_{\ell}=ka \widetilde{u}_{x}+O(a^2).
\end{equation*}

Then, the consistency analysis of energies yields:
\begin{equation*}
    \begin{split}
        E^{c,lin}&-E^{a,lin}\\
        &=\sum_{\ell=-M+1}^{M} \sum_{\substack{k=-N ,\\k\neq0}}^{N}\frac{a}{2}\left(\frac{k^{2}}{2}(u_{\ell}^{'})^{2}\phi_{xx}(k)\right)-\sum_{\ell=-M+1 }^{M} \sum_{\substack{k=-N, \\ k\neq0}}^{N}\frac{a}{2}\left(\frac{1}{2}(\frac{u_{\ell+k}-u_{\ell}}{a})^{2}\phi_{xx}(k)\right) \\
        &=\sum_{\ell=-M+1}^{M} \sum_{\substack{k=-N ,\\k\neq0}}^{N}\frac{a}{2}\left(\frac{k^{2}}{2}(\frac{u_{\ell+1}-u_{\ell}}{a})^{2}\phi_{xx}(k)\right)-\sum_{\ell=-M+1 }^{M} \sum_{\substack{k=-N, \\ k\neq0}}^{N}\frac{a}{2}\left(\frac{1}{2}(\frac{u_{\ell+k}-u_{\ell}}{a})^{2}\phi_{xx}(k)\right) \\
        &=\sum_{\ell=-M+1}^{M} \sum_{\substack{k=-N ,\\k\neq0}}^{N}\frac{a}{2}\left(\frac{k^{2}}{2}(\frac{a\widetilde{u}_{x}+O(a^{2})}{a})^{2}\phi_{xx}(k) \right)-\sum_{\ell=-M+1}^{M} \sum_{\substack{k=-N ,\\k\neq0}}^{N}\frac{a}{2}\left(\frac{1}{2}(\frac{ka\widetilde{u}_{x}+O(a^{2})}{a})^{2}\phi_{xx}(k) \right) \\
        &=\sum_{\ell=-M+1}^{M} \sum_{\substack{k=-N ,\\k\neq0}}^{N} \frac{a}{2}\frac{k^{2}}{2}\phi_{xx}(k)+O(a^{2})-\sum_{\ell=-M+1}^{M} \sum_{\substack{k=-N ,\\k\neq0}}^{N} \frac{a}{2}\frac{k^{2}}{2}\phi_{xx}(k)+O(a^{2}) \\
        &=O(a^{2}).
    \end{split}
\end{equation*}
}
\end{proof}
Next, we derive the formulae of forces for both \elaine{linear} atomistic and continuum models. Because the mesh is fully refined, the linearized continuum force of atom $\ell$ can be obtained from taking the first order variation of the linearized continuum energy \eqref{def_ContLinE} with respect to $u_{\ell}$, we thus get:
\begin{equation}\label{def_ContinuumF}
\begin{split}
    {F}_{\ell\elaine{,}}^{c,lin}(u)&:= \frac{1}{a}\frac{\delta E^{c,lin}(u)}{\delta u_{\ell}}= \frac{\delta \left[\sum_{j=-M+1}^{M}\left(\sum_{k=1}^{N}\frac{k^2}{2} \phi_{xx}(k)\right)\left(u_{j}^{'}\right)^{2}\right]}{\delta u_{\ell}} \\
    &=- \left(\left(\sum_{k=1}^{N} \frac{k^{2}}{2} \phi_{xx}(k)\right)\frac{2(u_{\ell+1}-u_{\ell})}{a^{2}}-\left(\sum_{k=1}^{N} \frac{k^{2}}{2} \phi_{xx}(k)\right)\frac{2(u_{\ell}-u_{\ell-1})}{a^{2}}\right) \\
    &=- \left(\sum_{k=1}^{N}k^{2}\phi_{xx}(k)\right)u_{\ell}^{''} ,
\end{split}
\end{equation}
where we recall the shorthand notation $u_{\ell}^{''}$ as
\begin{equation*}
u_{\ell}^{''}:=\frac{u_{\ell+1}-2u_{\ell}+u_{\ell-1}}{a^{2}}.
\end{equation*}
For the atomisitic forces, we recall $y_\ell=x_\ell+u_\ell$,
take the first order variation of the total atomistic energy \eqref{def_AtomicE} at atom $\ell$ and notice we employ the forward finite-differencing, hence, we obtain
\begin{equation}\label{def_AtomisticF}
    \begin{split}
        F_{\ell\elaine{,}}^{a}(u)&:= \frac{\delta E^{a,tot}}{\delta u_{\ell}}= \frac{\delta}{\delta u_{\ell}}\sum_{j =-M+1 }^{M}\sum_{\substack{k=-N,\\ k \neq{0}}}^{N}\frac{1}{2}\phi (\frac{y_{j+k}-y_{j}}{a})\\
        &=- \sum_{\substack{k=-N, \\k \neq{0}}}^{N}\frac{1}{2a}\left(\phi_{x}(k+\frac{u_{\ell+k}-u_{\ell}}{a})-\phi_{x}(k+\frac{u_{\ell}-u_{\ell-k}}{a})\right).
    \end{split}
\end{equation}
Linearizing the forces around the reference configuration by applying Taylor expansion to $\phi_{x}(\cdot)$, we obtain the linearized atomistic forces
\begin{equation}\label{def_AtomisticLinF}
    F_{\ell\elaine{,}}^{a,lin}(u):=-\sum_{\substack{k=-N,\\k \neq{0}}}^{N} \frac{1}{2}   \phi_{xx}(k)\left(\frac{u_{\ell+k}-2u_{\ell}+u_{\ell-k}}{a^{2}}\right)=-\sum_{k=1}^{N}  \phi_{xx}(k)\left(\frac{u_{\ell+k}-2u_{\ell}+u_{\ell-k}}{a^{2}}\right).
\end{equation}
In the next proposition, we summarize the consistency errors between the linearized atomisitic and linearized continuum forces.
\begin{proposition} [Consistency analysis of force]\label{proposition:forceconsistency}
Given a fully refined continuum mesh on the 1D atomistic chain, the linearized atomistic force equation \eqref{def_AtomisticLinF} for atom $\ell$ is
\begin{equation*}
    F_{\ell\elaine{,}}^{a,lin}(u):=-\sum_{k=1}^{N}  \phi_{xx}(k)\left(\frac{u_{\ell+k}-2u_{\ell}+u_{\ell-k}}{a^{2}} \right)
\end{equation*}
and the linearized continuum force equation \eqref{def_ContinuumF} for node $\ell$ is
\begin{equation*}
    F_{\ell\elaine{,}}^{c,lin}(u)=-\left(\sum_{k=1}^{N}k^{2}\phi_{xx}(k)\right)u_{\ell}^{''}.
\end{equation*}
Thus, the consistency error between \eqref{def_AtomisticLinF} and \eqref{def_ContinuumF} is $O(a^2)$.
\end{proposition}
\begin{proof}
Comparing $F^{c,lin}$ and $F^{a, lin}$, we have
\begin{equation*}
        F_{\elaine{\ell,}}^{c,lin}(u)-_{\elaine{\ell,}}F^{a,lin}(u)=-\sum_{k=1}^{N}\phi_{xx}(k)\left(k^{2}u_{\ell}^{''}\right)+\sum_{k=1}^{N} \phi_{xx}(k)\left(\frac{u_{\ell+k}-2u_{\ell}+u_{\ell-k}}{a^2}\right).
\end{equation*}
For any $k=2,\dots,N$, we compare $u_{\ell+k}$ and $u_{\ell-k}$ around $\ell$. \elaine{
Recall the quintic spline interpolation $\widetilde{u}$ defined in \eqref{quintic_spline}, we have
\begin{equation*}
    \begin{split}
        u_{\ell+k}&= \widetilde{u}\left(a(\ell+k)\right)-\widetilde{u}\left(a\ell\right)+\widetilde{u}\left(a\ell\right)=
        u_{\ell}+ka \widetilde{u}_{x}+\frac{1}{2}(ka)^2\widetilde{u}_{xx}+\frac{1}{6}(ka)^3\elaine{\widetilde{u}_{xxx}}+O(a^4), \\
        u_{\ell-k}&= \widetilde{u}\left(a(\ell-k)\right)-\widetilde{u}\left(a\ell\right)+\widetilde{u}\left(a\ell\right)= u_{\ell}-ka \widetilde{u}_{x}+\frac{1}{2}(ka)^2\widetilde{u}_{xx}-\frac{1}{6}(ka)^3\widetilde{u}_{xxx}+O(a^4).
    \end{split}
\end{equation*}
}

\noindent \elaine{Utilizing} this Taylor expansion \elaine{for} the atomistic \elaine{and continuous} linear force equation\elaine{s}, the consistency analysis yields
\begin{equation*}
    \begin{split}
        F_{\elaine{\ell,}}^{c,lin}&(u)-F_{\elaine{\ell,}}^{a,lin}(u)\\&= -\sum_{k=1}^{N}\phi_{xx}(k)\left(k^{2}u_{\ell}^{''}\right)
        +\sum_{k=1}^{N}\phi_{xx}(k)(\frac{u_{\ell+k}-2u_{\ell}+u_{\ell-k}}{a^2}) \\
        &\elaine{=} -\sum_{k=1}^{N}\phi_{xx}(k)\big(k^2\frac{u_{\ell+1}-2u_{\ell}+u_{\ell-1}}{a^{2}} \big)+\sum_{k=1}^{N} \phi_{xx}(k)\big(k^2\elaine{\widetilde{u}_{xx}}+O(a^2)\big) \\
        &=\elaine{ -\sum_{k=1}^{N}\phi_{xx}(k)\big(k^{2}\widetilde{u}_{xx}+O(a^2)\big)+\sum_{k=1}^{N}\phi_{xx}(k)\big(k^2 \widetilde{u}_{xx}+O(a^2)\big)}\\
    &= O(a^2).
    \end{split}
\end{equation*}
\elaine{A more thorough proof can be seen in Subsection \ref{RigorProof}.}
\end{proof}
\begin{remark}
Notice that $a$ is chosen to be $\frac{1}{M}$ with $M$ being large, hence, the consistency error becomes small when the number of atoms within $\Omega$ is sufficiently large.
\end{remark}

\subsection{Derivation of a Symmetric Blending Model for the AtC Coupling}
In this section we will derive a symmetric and consistent force-based atomistic-to-continuum scheme for the 1D atomistic chain.

We first divide the domain of interest into three distinct sub-domains: $\Omega^{a}$: the domain described by the atomistic force; $\Omega^{c}$: the domain described by the continuum force; and $\Omega^{b}$: the blending region where the atomistic and continuum force models are both used.

We  now introduce a smooth blending function $\beta$ that can be defined as such:
\begin{definition}[Definition of blending function]\label{def_blendfunc}
We may define a smooth blending function $\beta_{\ell}$ such that:
\begin{equation}\label{blendingfunction}
    \beta_{\ell}=\begin{cases}
    1, \quad &\ell \in \Omega^{a} \\
    0, \quad &\ell \in \Omega^{c} \\
    \in(0,1), \quad &\ell \in \Omega^{b}.
    \end{cases}
\end{equation}
\end{definition}
This blending function can take many forms and we will employ linear spline, cubic spline, and quintic spline blending functions in the numerical experiments in Section 4.

Notice that creating a  linearized force equation will give way to easier analysis in studying the stability of the scheme and providing insights on the coupling conditions of more general cases, so we focus on blending the linearized atomistic and continuum forces. \helen{In order for consistent symmetry, we start from $k=-N,\dots, N,$ and $k\neq 0$.} Consequently, we start from the linearized, atomistic force equation in \eqref{def_AtomisticLinF} and incorporate the blending function $\beta_{\ell}$ as follows:
\begin{align*}
     &  F_{\ell,}^{a,lin}:=-\sum_{k=1}^{N}  \phi_{xx}(k)\frac{u_{\ell+k}-2u_{\ell}+u_{\ell-k}}{a^{2}}=-\sum_{\substack{k=-N,\\k \neq{0}}}^{N} \frac{1}{2} \phi_{xx}(k)\frac{u_{\ell+k}-2u_{\ell}+u_{\ell-k}}{a^{2}} \nonumber \\
    &\,= -\sum_{\substack{k=-N, \\k \neq{0}}}^{N}\left(\frac{\beta_{\ell}+\beta_{\ell+k}}{2}\right)\frac{1}{2}\phi_{xx}(k)\frac{u_{\ell+k}-2u_{\ell}+u_{\ell-k}}{a^{2}}-\sum_{\substack{k=-N, \\k \neq{0}}}^{N}\left(1-\frac{\beta_{\ell}+\beta_{\ell+k}}{2}\right)\frac{1}{2}\phi_{xx}(k)\frac{u_{\ell+k}-2u_{\ell}+u_{\ell-k}}{a^{2}},
    \end{align*}
   \helen{such that the term $\frac{u_{\ell+|k|}-2u_{\ell}+u_{\ell-|k|}}{a^2}$ is multiplied by the pair $\left(\frac{\beta_{\ell}+\beta_{\ell+|k|}}{2}\right)$ and $\left(\frac{\beta_{\ell}+\beta_{\ell-|k|}}{2}\right)$, respectively.}
Next, we further simplify and get
    \begin{align}
    \label{def_AtomisticFLinBlend}
     F_{\ell,}^{a,lin}
    =&-\sum_{k=1}^{N}\left(\frac{\beta_{\ell-k}+2\beta_{\ell}+\beta_{\ell+k}}{4}\right)\phi_{xx}(k)\frac{u_{\ell+k}-2u_{\ell}+u_{\ell-k}}{a^{2}}\nonumber\\
    &\qquad -\sum_{k=1}^{N}\left(1-\frac{\beta_{\ell-k}+2\beta_{\ell}+\beta_{\ell+k}}{4}\right)\phi_{xx}(k)\frac{u_{\ell+k}-2u_{\ell}+u_{\ell-k}}{a^{2}}.
    \end{align}
Then,  we approximate the second term of the equation by using the linearized continuum portion. Therefore, we get the blended force function which is defined as:
\begin{align}
\label{def_linear_bqcf}
   & F_{\ell,}^{bqcf,lin}(u)
    :=-\sum_{k=1}^{N}\left(\frac{\beta_{\ell-k}+2\beta_{\ell}+\beta_{\ell+k}}{4}\right)\phi_{xx}(k)\frac{u_{\ell+k}-2u_{\ell}+u_{\ell-k}}{a^{2}}\nonumber\\
    &\qquad\qquad\qquad\quad -\sum_{k=1}^{N}\left(1-\frac{\beta_{\ell-k}+2\beta_{\ell}+\beta_{\ell+k}}{4}\right)\phi_{xx}(k)k^{2}\frac{u_{\ell+1}-2u_{\ell}+u_{\ell-1}}{a^{2}} \\
    &\;=-\sum_{k=1}^{N}\left(\frac{\beta_{\ell-k}+2\beta_{\ell}+\beta_{\ell+k}}{4}\right)\phi_{xx}(k)\frac{u_{\ell+k}-2u_{\ell}+u_{\ell-k}}{a^{2}}-\sum_{k=1}^{N}\left(1-\frac{\beta_{\ell-k}+2\beta_{\ell}+\beta_{\ell+k}}{4}\right)\phi_{xx}(k)k^{2}u_{\ell}^{''}.\nonumber
\end{align}
\begin{remark}
As $\ell-k$ and $\ell+k$ are both employed in the term of the blending function, this blending operator is symmetric as in \cite{seleson2013aa}.

We also see that the first term recovers \eqref{def_AtomisticLinF} for $\beta \equiv 0$, and the second term recovers \eqref{def_ContinuumF} for $\beta\equiv 1$. Thus, the consistency error between the linearized $F_{\elaine{\ell,}}^{bqcf,lin}$ and $F_{\elaine{\ell,}}^{a,lin}$ is also of $O(a^2)$.
\end{remark}


\section{Stability Analysis for the Linearized Blending Model}

\subsection{Bilinear Form of Linearized Blending Model}
In this section we study the size of the blending region with respect to the $H^{1}$ stability of the blending operator. This is achieved by obtaining the optimal conditions in which the linearized coupling operator is positive definite under the discrete $H^{1}$ semi-norm.

From \eqref{def_linear_bqcf}, we consider the bilinear form
 \begin{equation}\label{def_bilinearform}
     \<F_{,}^{bqcf,lin}(u), v\> = \<F_{\elaine{,}1}^{bqcf,lin}(u), v\>+\sum_{k=2}^{N} \<F_{\elaine{,}k}^{bqcf,lin}(u), v\>,\quad \forall v\in\mathcal{U},
 \end{equation}
where the first neighbor interaction---the first term---is set apart due to its simplicity in analysis. The second term accounts for the next-nearest neighbor to the $N$-th nearest neighbor.

To observe the stability of the operator, we first look at the nearest and next-nearest neighbor interaction, for simplicity as well as \helen{for the coercivity assumption on $\phi_{xx}(1)>0$ and $\phi_{xx}(k)\le 0$ with $k\ge 2$}, to find the constraints on the size of  blending region. Then, we discuss how the next-nearest neighbor analysis can be extended to the general $N$-th neighbor interaction.

The discrete stability analysis is inspired by and similar to the analogous continuous analysis for the force-based operator that can be seen in the Appendix~ \ref{appendix}. We proceed with the analysis for the discrete case.

\subsection{Stability Analysis for  Next-Nearest Neighbor Interaction Range: $N=2$.}
\begin{lemma}\label{lemma:operators} For any displacements $u=(u_{\ell})_{\ell =-M+1}^{M}$ from the deformed configuration $y_{\ell}=x_\ell+u_\ell$, the bilinear forms of nearest neighbor and the next-nearest neighbor interaction operator can be written as
\begin{equation}
    \begin{split}
    \label{def_NeighborInteraction}
        &<F_{\elaine{,}1}^{bqcf,lin}(u),u>= \phi_{xx}(1)\norm{u^{'}}^{2}_{\ell_{2}} \\
        &<F_{\elaine{,}2}^{bqcf,lin}(u),u>=
       2\left\{ 2\phi_{xx}(2)\norm{u^{'}}_{\ell_{2}}^2-\frac{\phi_{xx}(2)}{2}a^{2}\norm{\sqrt{\beta}u^{''}}_{\ell_{2}}^{2}+\frac{\phi_{xx}(2)}{2}a^{2}\norm{\sqrt{|\beta^{''}|}u'}_{\ell_{2}}^{2}+R+S\right\}.
    \end{split}
\end{equation}
where \[R:=\sum_{\ell=-M+1}^{M}\frac{\phi_{xx}(2)}{2}\left(u_{\ell}^{'}\beta_{\ell-1}^{(3)}(u_{\ell-1})-u^{''}_{\ell}\beta^{''}_{\ell}u^{'}_{\ell}a\right)a^{3}\] and \[S:=\sum_{\ell=-M+1}^{M}\frac{\phi_{xx}(2)}{2}(u_{\ell}^{'})a^{2}\left(\beta'_{\ell}u_{\ell}^{'}-\beta_{\ell-2}^{'}u_{\ell-2}^{'}\right).\]
\end{lemma}

\begin{proof}
For the nearest neighbor interaction, as the $F^{a,lin}$ and $F^{c,lin}$ coincide, we have
\begin{equation*}\label{def_NearestOp}
\begin{split}
<F_{\elaine{,}1}^{bqcf,lin}(u),u> &=-\sum_{\ell=-M+1}^{M}\bigg[\phi_{xx}(1)a\left(\frac{u_{\ell+1}-2u_{\ell}+u_{\ell-1}}{a^{2}}\right) \bigg] u_{\ell} \\
&=-\sum_{\ell=-M+1}^{M}\phi_{xx}(1)u_{\ell}^{''}u_{\ell}a.
\end{split}
\end{equation*}
Then, using the discrete derivative and summation by parts formula from Lemma \ref{lemma:summation}, we conclude that
\begin{equation*}\label{def_Nearest}
    \begin{split}
        <F_{\elaine{,}1}^{bqcf,lin}(u),u>
        &=-\sum_{\ell=-M+1}^{M}\phi_{xx}(1)u_{\ell}^{''}u_{\ell} a
        =\phi_{xx}(1)\norm{u^{'}}_{\ell_{2}}^{2}.
    \end{split}
\end{equation*}
For, the next-nearest neighbor interaction, recall from \eqref{def_linear_bqcf} for the linearized blending operator $F^{bqcf,lin}_{\ell,2}$
\[
\begin{split}
F_{\ell,2}^{bqcf, lin}(u)
=&
-\big[\left(\frac{\beta_{\ell+k}+2\beta_{\ell}+\beta_{\ell-k}}{4}\right)
\phi_{xx}(k)\left(\frac{u_{\ell+k}-2u_{\ell}+u_{\ell-k}}{a^2}\right)\\
&\qquad+\left(1-\frac{\beta_{\ell+k}+2\beta_{\ell}+\beta_{\ell-k}}{4}\right)
\phi_{xx}(k)\left(\frac{u_{\ell+1}-2u_{\ell}+u_{\ell-1}}{a^2}\right)k^2 \big]\bigg|_{k=2},
\end{split}
\]
so, the bilinear form with test function $u$ is
\begin{equation}\label{bilinear_2nd}
\begin{aligned}
\< F^{bqcf, lin}_{\elaine{,}2}(u), u\>
=& \sum_{\ell=-M+1}^{M}  F^{bqcf, lin}_{\ell,2}(u) \cdot u_{\ell}a\\
=& \sum_{\ell=-M+1}^{M}\bigg\{ -
\left(\frac{\beta_{\ell+2}+2\beta_{\ell}+\beta_{\ell-2}}{4}\right)
\phi_{xx}(2)\left(\frac{u_{\ell+2}-2u_{\ell}+u_{\ell-2}}{a^2}\right)a\\
&\qquad -\left(1-\frac{\beta_{\ell+2}+2\beta_\ell+\beta_{\ell-2}}{4}\right)
\phi_{xx}(2)\left(\frac{u_{\ell+1}-2u_{\ell}+u_{\ell-1}}{a^2}\right)4a\bigg\}\cdot u_{\ell}.
\end{aligned}
\end{equation}
We particularly focus on terms contributed to $\beta_{\ell}$ as the other terms could be similarly treated. Hence, we divide the constant `1' in \eqref{bilinear_2nd} by $1=\frac{1}{4}+\frac{2}{4}+\frac{1}{4}$, collect terms contributed to $\beta_{\ell}$ and recall the finite difference defined in \eqref{def_FiniteDiff}, then we have
\begin{align}
    T:=&\frac{-\phi_{xx}(2)}{2}\bigg[
    \sum_{\ell=-M+1}^{M} \left(\frac{u_{\ell+2}-2u_{\ell}+u_{\ell-2} }{a}\right)\left(\beta_{\ell} u_{\ell}\right)
    +\left(1-\beta_{\ell}\right)
   \left(\frac{u_{\ell+1}-2u_{\ell}+u_{\ell-1}}{a}\right)4  u_{\ell}\bigg]\nonumber \\
    =& \frac{-\phi_{xx}(2)}{2}\bigg[\sum_{\ell=-M+1}^{M} \left(u^{'}_{\ell+1}+u^{'}_{\ell}-u^{'}_{\ell-1}-u^{'}_{\ell-2}\right)\left(\beta_{\ell} u_{\ell}\right)
    +\left(1-\beta_{\ell}\right)
    \left(u^{'}_{\ell}-u^{'}_{\ell-1}\right)4  u_{\ell} \bigg]\nonumber\\
    =&\frac{-\phi_{xx}(2)}{2}\bigg[\sum_{\ell=-M+1}^{M}\left(u^{'}_{\ell}-u^{'}_{\ell-1}\right)4u_{\ell}\bigg]\nonumber\\
   &\qquad +\frac{-\phi_{xx}(2)}{2}\sum_{\ell=-M+1}^{M}\left[
    \left(u^{'}_{\ell+1}-u^{'}_{\ell}\right)-2\left(u^{'}_{\ell}-u^{'}_{\ell-1}\right)+\left(u^{'}_{\ell-1}-u^{'}_{\ell-2}\right)
    \right]\cdot\big(\beta_{\ell} u_{\ell} \big)\nonumber\\
=:& T_1+T_2.\label{def_T1T2}
\end{align}
We consider $T_1$ first by using Lemma~\ref{lemma:summation}
\begin{equation}\label{eq:T1}
\begin{aligned}
    T_1&=\sum_{\ell=-M+1}^{M}\frac{-\phi_{xx}(2)}{2}
    \left(u^{'}_{\ell}-u^{'}_{\ell-1}\right)4u_{\ell}
    =\sum_{\ell=-M+1}^{M}\frac{\phi_{xx}(2)}{2}
    u^{'}_{\ell-1}4\left(u_{\ell}-u_{\ell-1}\right)\\
    &=\sum_{\ell=-M+1}^{M}\frac{\phi_{xx}(2)}{2}
    u^{'}_{\ell-1}4\left(u^{'}_{\ell-1}\right)a
    =2\phi_{xx}(2)\|u^{'}\|_{\ell_{2}}^2.
    \end{aligned}
\end{equation}
Treating $T_2$ is more tedious and is still mainly based on Lemma~\ref{lemma:summation}. We have
\begin{align*}
    T_2=&\sum_{\ell=-M+1}^{M}\frac{-\phi_{xx}(2)}{2}\left[
    \left(\big(u^{'}_{\ell+1}-u^{'}_{\ell}\big)-\big(u^{'}_{\ell}-u^{'}_{\ell-1}\big)\right)-
    \left(\big(u^{'}_{\ell}-u^{'}_{\ell-1}\big)-
    \big(u^{'}_{\ell-1}-u^{'}_{\ell-2}\big)\right)
    \right]\cdot\big(\beta_{\ell} u_{\ell} \big)\\
    =& \sum_{\ell=-M+1}^{M}\frac{\phi_{xx}(2)}{2}\left[
    \left(\big(u^{'}_{\ell+1}-u^{'}_{\ell}\big)-\big(u^{'}_{\ell}-u^{'}_{\ell-1}\big)\right)  \right]\cdot\big(\beta_{\ell+1} u_{\ell+1}-\beta_{\ell}u_{\ell} \big)\\
  =& \sum_{\ell=-M+1}^{M}\frac{\phi_{xx}(2)}{2}
    \big(u^{''}_{\ell}-u^{''}_{\ell-1}\big)a\cdot\big(\beta_{\ell+1} u_{\ell+1}-\beta_{\ell}u_{\ell} \big) \\
=& \sum_{\ell=-M+1}^{M}\frac{-\phi_{xx}(2)}{2}
    \big(u^{''}_{\ell}\big)a\cdot\left[\big(\beta_{\ell+1} u_{\ell+1}-\beta_{\ell}u_{\ell} \big)-\big(\beta_{\ell} u_{\ell}-\beta_{\ell-1}u_{\ell-1} \big)\right]\\
=&   \sum_{\ell=-M+1}^{M}\frac{-\phi_{xx}(2)}{2}
    \big(u^{''}_{\ell}\big)a\cdot\left[\left(\beta_{\ell+1} u_{\ell+1}-\beta_{\ell} u_{\ell+1}\right)+\left(\beta_{\ell} u_{\ell+1}-2\beta_{\ell}u_{\ell} +\beta_{\ell}u_{\ell-1}\right)-\left(\beta_{\ell}u_{\ell-1}-\beta_{\ell-1}u_{\ell-1} \right)\right]\\
=& \sum_{\ell=-M+1}^{M}\frac{-\phi_{xx}(2)}{2}
    \big(u^{''}_{\ell}\big)\cdot\beta_{\ell} u^{''}_{\ell}a^3
    +\frac{-\phi_{xx}(2)}{2}
    \big(u^{''}_{\ell}\big)\cdot\left[\beta^{'}_{\ell} u_{\ell+1}-\beta_{\ell-1}^{'}u_{\ell-1}\right]a^2\\
=&:-\frac{\phi_{xx}(2)}{2}a^2\|\sqrt{\beta} u^{''}\|_{\ell_{2}}^2
+ T_{22}.
\end{align*}
Now we mainly focus on $T_{22}$ term, which can be treated as
\[
\begin{aligned}
T_{22}=&\sum_{\ell=-M+1}^{M}\frac{-\phi_{xx}(2)}{2}
    \big(u^{'}_{\ell}-u^{'}_{\ell-1}\big)\cdot\left[\beta^{'}_{\ell} u_{\ell+1}-\beta_{\ell-1}^{'}u_{\ell-1}\right]a\\
=&\sum_{\ell=-M+1}^{M}\frac{\phi_{xx}(2)}{2}
    \big(u^{'}_{\ell}\big)\cdot\left[\left(\beta^{'}_{\ell} u_{\ell+1}-\beta_{\ell-1}^{'}u_{\ell-1}\right)-\left(\beta^{'}_{\ell-1} u_{\ell}-\beta_{\ell-2}^{'}u_{\ell-2}\right)\right]a.
    \end{aligned}
\]
We work on
\[
\begin{aligned}
&\left(\beta^{'}_{\ell} u_{\ell+1}-\beta_{\ell-1}^{'}u_{\ell-1}\right)-\left(\beta^{'}_{\ell-1} u_{\ell}-\beta_{\ell-2}^{'}u_{\ell-2}\right)\\
&\;=\beta^{'}_\ell(u_{\ell+1}-u_{\ell-1})
+u_{\ell-1}\left(\beta^{'}_{\ell}-2\beta^{'}_{\ell-1}+\beta^{'}_{\ell-2}\right)
+\beta^{'}_{\ell-1}(u_{\ell-1}-u_{\ell})
+\beta^{'}_{\ell-2}(u_{\ell-2}-u_{\ell-1})\\
&\,=\beta^{'}_{\ell}(u'_{\ell}+u'_{\ell-1})a+u_{\ell-1}\beta_{\ell-1}^{(3)}a^2-\beta^{'}_{\ell-1}u'_{\ell-1}a-\beta'_{\ell-2}u'_{\ell-2}a.
\end{aligned}
\]
Then we plug into $T_{22}$ to get
\[
\begin{aligned}
T_{22}=&\sum_{\ell=-M+1}^{M}\frac{\phi_{xx}(2)}{2}    \big(u^{'}_{\ell}\big)\cdot\bigg[\beta^{'}_{\ell}(u'_{\ell}+u'_{\ell-1})a+u_{\ell-1}\beta_{\ell-1}^{(3)}a^2-\beta^{'}_{\ell-1}u'_{\ell-1}a-\beta'_{\ell-2}u'_{\ell-2}a\bigg]a\\
=&\frac{\phi_{xx}(2)}{2}\sum_{\ell=-M+1}^{M}\big(u'_{\ell}\big)\left( \bigg[u_{\ell-1}\beta^{(3)}_{\ell-1}a^2+\left(\beta'_{\ell}u'_{\ell-1}-\beta'_{\ell-1}u'_{\ell-1}\right)a\bigg]a
+ \bigg[\beta'_{\ell}u'_{\ell}-\beta'_{\ell-2}u'_{\ell-2}\bigg]a^2\right) \\
=&\frac{\phi_{xx}(2)}{2}\sum_{\ell=-M+1}^{M}\bigg(\big(u^{'}_{\ell}\big)
\beta^{(3)}_{\ell-1}\left(u_{\ell-1}\right)+\big(u'_{\ell}\big)^2\beta^{''}_{\ell}
-u''_{\ell}\beta^{''}_{\ell}u'_{\ell}a \bigg)a^3\\
&\,+\frac{\phi_{xx}(2)}{2}
\sum_{\ell=-M+1}^{M}
\bigg(\beta'_{\ell}u'_{\ell}
-\beta'_{\ell-2}u'_{\ell-2}\bigg)u'_{\ell}a^2.
\end{aligned}
\]
Summarizing all terms $T_{1}$, $T_{2}$, $T_{22}$ for terms belong to $\beta_{\ell}$ in \eqref{bilinear_2nd}, and treat those for $\beta_{\ell-2}$ and $\beta_{\ell+2}$ in a similar way, we get \eqref{def_NeighborInteraction} and $R$, $S$ terms.
\end{proof}
\begin{remark}
The terms $R$ and $S$ from above are viewed as ``residual term''. Thus, we will estimate their bounds controlled by the the support of $\beta'_{\ell}$, the size of blending region, in pursuit of the positive-definiteness of the bilinear form.
\end{remark}


\begin{lemma}\label{lemma:errorbound} Let $R$, $S$ be defined as above, then we have the following estimates
\begin{equation}
    \begin{split}
    \label{def_ErrorBounds}
        &\abs{R}\leq \frac{\left(-\phi_{xx}(2)\right)}{2}\bigg(c_{3}L^{-\frac{5}{2}}a^{-\frac{1}{2}}
        +2c_{2}(L)^{-2}\bigg)\norm{u^{'}}_{\ell_{2}}^{2}, \\
        &\abs{S}\leq
         \big(-\phi_{xx}(2)\big) c_1(L)^{-1}\|u'\|_{\ell_2}^2,
    \end{split}
\end{equation}
where $L$ is the number of atoms within the blending region $\Omega_b$, $a=\frac{1}{M}$ being the lattice spacing, and $2M$ being the total number of atoms within the periodic domain $\Omega =(-1, 1]$, \elaine{and the constants $c_1$, $c_2$ and $c_3$ depend on $\beta^{(j)}$ with $j=1,\dots, 3$, respectively.}
\end{lemma}
\begin{proof}
Recall that \[R=\sum_{\ell=-M+1}^{M}\frac{\phi_{xx}(2)}{2}\left(u_{\ell}^{'}\beta_{\ell-1}^{(3)}(u_{\ell-1})-u^{''}_{\ell}\beta^{''}_{\ell}u^{'}_{\ell}a\right)a^{3}.\]
\helen{Also notice that the finite differences of $\beta$ are nonzero only on $\Omega_b$}, so utilizing the fact that $\phi_{xx}(2)\leq0$ and H{\"o}lder's Inequality,  we get
\begin{equation*}
    \begin{split}
        \abs{R}&\le \abs{\sum_{\ell \in \Omega^{b}}\frac{\phi_{xx}(2)}{2}(u_{\ell}^{'})\left(\beta_{\ell}^{(3)}(u_{\ell-1})\right)a^{3}}
        +\abs{\sum_{\ell \in \Omega^{b}}
        \frac{\phi_{xx}(2)}{2}u''_{\ell}\beta''_{\ell}u'_{\ell}a^4} \\
        &\leq \frac{-\phi_{xx}(2)}{2}\norm{\beta^{(3)}}_{\ell^\infty}\norm{u^{'}}_{\ell_{2}(\Omega^{b})} \norm{u}_{\ell_{2}(\Omega^{b})}a^{2}+\frac{-\phi_{xx}(2)}{2}\|\beta^{(2)}\|_{\ell^{\infty}}\|u''\|_{\ell_2(\Omega^b)}\, \|u'\|_{\ell_2(\Omega^b)}a^3.
    \end{split}
\end{equation*}
Also notice that $|\Omega_b|=La$ and by the discrete  Poincar{\'e} Inequality,  we have $\norm{u}_{\ell_{2(\Omega^{b})}}\leq(La)^{\frac{1}{2}}\norm{u^{'}}_{\ell_{2}}$.
In addition, we use the fact that $\norm{\beta^{(j)}}_{\ell^\infty}\leq c \,(La)^{-j}$. Then
\begin{equation*}
    \begin{split}
        \abs{R} &\leq 
        \frac{-\phi_{xx}(2)}{2}\norm{\beta^{(3)}}_{\ell^\infty}\norm{u^{'}}_{\ell_{2}(\Omega^{b})} \norm{u}_{\ell_{2}(\Omega^{b})}a^{2}+\frac{-\phi_{xx}(2)}{2}\|\beta^{(2)}\|_{\ell^{\infty}}\|u''\|_{\ell_2(\Omega^b)}\, \|u'\|_{\ell_2(\Omega^b)}a^3\\
        & \leq \frac{-\phi_{xx}(2)}{2}c_{3}(La)^{-3}(La)^{\frac{1}{2}}\norm{u^{'}}_{\ell_2}^{2}a^{2}
        +\frac{-\phi_{xx}(2)}{2}c_{2}(La)^{-2}\frac{2}{a}\|u'\|^2_{\ell_2 }a^3
        \\
        & \leq \frac{-\phi_{xx}(2)}{2}\bigg(c_{3}L^{-\frac{5}{2}}a^{-\frac{1}{2}}
        +2c_{2}(L)^{-2}\bigg)\norm{u^{'}}_{\ell_{2}}^{2}.
    \end{split}
\end{equation*}
Next, we bound $S$. Recall that  \[S:=\sum_{\ell=-M+1}^{M}\frac{\phi_{xx}(2)}{2}(u_{\ell}^{'})a^{2}\left(\beta'_{\ell}u_{\ell}^{'}-\beta_{\ell-2}^{'}u_{\ell-2}^{'}\right).\] Utilizing similar inequalities, we obtain
\begin{equation*}
    \begin{split}\label{def_SBound}
        \abs{S} &
        \leq \frac{-\phi_{xx}(2)}{2}\|\beta'\|_{\ell^\infty}\cdot 2\|u'\|_{\ell_2}^2a
        \leq\frac{-\phi_{xx}(2)}{2} c_1(La)^{-1}\cdot 2\|u'\|_{\ell_2}^2a= \big(-\phi_{xx}(2)\big) c_1(L)^{-1}\|u'\|_{\ell_2}^2.
    \end{split}
\end{equation*}
Hence, we prove the lemma.
\end{proof}

\begin{theorem}\label{theorem:blendingregion} \elaine{Suppose that the number of atoms $M$ within the chain model is sufficiently large, or equivalently the lattice spacing $a=\frac{1}{M}$ is sufficiently small; also suppose that the fully atomistic model with next-nearest-neighbor interaction $N=2$, is stable so that $[\phi_{xx}(1)+4\phi_{xx}(2)]>0$.} Let $L$ denote the number of atoms within the blending region and let the blending function $\beta$ be  sufficiently smooth such that $\norm{\beta^{(j)}}_{\infty}\leq(La)^{-j}$. Then there exists a positive constant \elaine{$\widetilde{C}$} such that
\begin{equation}
    \elaine{
   \helen{\sum_{k=1}^{2} \sum_{\ell=-M+1}^{M}
    \<F^{bqcf,lin}_{\ell,k}}u,u\>  } \geq 
    \elaine{\widetilde{C}}\norm{u^{'}}_{\ell_{2}}^{2}
\end{equation}
\elaine{$\widetilde{C}$} is strictly bounded above zero as $a=\frac{1}{M}\rightarrow0$.
\end{theorem}
\begin{proof}
\elaine{For $N=2$ and from Lemma \ref{lemma:operators}}, the blended force-based operator satisfies
\begin{equation}\label{def_operatorBounds}
\begin{aligned}
\elaine{\sum_{\ell=-M+1}^{M}<F_{\ell,2}^{bqcf,lin}(u),u>}= &
        \left(\phi_{xx}(1)+4\phi_{xx}(2)\right)\norm{u^{'}}_{\ell_{2}}^2-{\phi_{xx}(2)}a^{2}\norm{\sqrt{\beta}u^{''}}_{\ell_{2}}^{2}+{\phi_{xx}(2)}a^{2}\norm{\sqrt{\beta^{''}}u^{'}}_{\ell_{2}}^{2}\\
      & \qquad \qquad +2R+2S.
        \end{aligned}
\end{equation}

From Lemma \ref{lemma:errorbound}, we have
\begin{equation}\label{inequalities}
\begin{split}
    \abs{2R+2S}
    &\elaine{\leq} -\frac{\phi_{xx}(2)}{2}\left(\elaine{2c_{3}}L^{-\frac{5}{2}}a^{-\frac{1}{2}}+\elaine{4c_{2}}L^{-2}+\elaine{c_{1}}L^{-1}\right)\norm{u^{'}}_{\ell_{2}}^{2} \\
    &\elaine{\leq} -\frac{\phi_{xx}(2)}{2}\left(\elaine{c_{4}}L^{-\frac{5}{2}}a^{-\frac{1}{2}}\right)\norm{u^{'}}_{\ell_{2}}^{2},
    \end{split}
\end{equation}
hence, we have
\[
2R+2S\ge \frac{\phi_{xx}(2)}{2}\left(\elaine{c_{4}}L^{-\frac{5}{2}}a^{-\frac{1}{2}}\right)\norm{u^{'}}_{\ell_{2}}^{2}
\]
with the latter inequality following from $L^{-2}\le L^{-1} \leq L^{-\frac{5}{2}}a^{-\frac{1}{2}}$ \elaine{as $L^{-\frac{5}{2}}a^{-\frac{1}{2}}$ dominates the latter two terms when $a$ is sufficiently small. 
}

From \eqref{def_operatorBounds}, we want to observe the terms that do not favor the coercivity in the terms of the $H^{1}$ semi-norm.
For the second term, since  \elaine{$\phi_{xx}(2)\le 0$}, we have
\begin{equation*}
    -{\phi_{xx}(2)}a^{2}\norm{\sqrt{\beta}u^{''}}_{\ell_{2}}^{2} \geq 0,
\end{equation*} and thus it does not negatively contribute.
For the third term we observe that
\begin{equation*}
    {\phi_{xx}(2)}a^{2}\norm{\sqrt{\beta^{''}}u^{'}}_{\ell_{2}}^{2} \elaine{\geq} {\phi_{xx}(2)}\elaine{c_{5}}L^{-2}\norm{u^{'}}_{\ell_{2}}^{2}
\end{equation*}
because of  $\phi_{xx}(2) \leq 0$. \elaine{For up to $N=2$ neighbour interaction, we have from the coercivity of atomistic model that $\phi''(1)+4\phi''(2)>0$, hence, for the linearized force-blended model, we have}
\begin{equation*}
    \<F_{\elaine{,}1}^{bqcf,lin}+F_{\elaine{,}2}^{bqcf,lin},u\>
 \elaine{\geq} \left[\phi_{xx}(1)+{\phi_{xx}(2)}\elaine{\left( 4+c_4 L^{-\frac{5}{2}}a^{-\frac{1}{2}}+c_5 L^{-2}\right)} \right]\norm{u^{'}}_{\ell_{2}}^{2}
\elaine{\geq \widetilde{C}}\norm{u^{'}}^{2}_{\ell_{2}}
\end{equation*}
for \elaine{$\widetilde{C}=\left(\phi_{xx}(1)+2\phi_{xx}(2)\right)>0$ } strictly positive, and independent of $a\rightarrow 0$.

\noindent
Thus, we can conclude that the necessary and optimal blending size for coercivity is \elaine{$L^{-\frac{5}{2}}a^{-\frac{1}{2}}\lesssim 1/c_4$ that is $L= \widehat{C}a^{-\frac{1}{5}}=\widehat{C} M^{\frac{1}{5}}$ for some $\widehat{C}>0$.} 
\end{proof}
\subsection{\elaine{Stability Analysis for general $N$-th Nearest Neighbor Interaction Range}}
To observe the general form of the $N$-th neighbor interaction range, we notice that for $k=3,\dots,N$
\begin{equation}\label{bqcf_bilinear_genN}
\begin{aligned}
\< F^{bqcf, lin}_{,k}&(u), u\>\\
= &  -\phi_{xx}(k)\sum_{\ell=-M+1}^{M}\bigg\{  \left(\frac{\beta_{\ell+k}+2\beta_{\ell}+\beta_{\ell-k}}{4}\right)
\left(\frac{u_{\ell+k}-2u_{\ell}+u_{\ell-k}}{a^2}\right)\\
&\qquad +\left(1-\frac{\beta_{\ell+k}+2\beta_{\ell}+\beta_{\ell-k}}{4}\right)
k^2\left(\frac{u_{\ell+1}-2u_{\ell}+u_{\ell-1}}{a^2}\right) \bigg\}u_{\ell}a.
\end{aligned}
\end{equation}
The $k$-th neighbor interaction differs by having $\beta_{\ell\pm k}$ terms which are treated similarly to $\beta_{\ell}$ terms.  $\phi_{xx}(k)$ is a non-positive constant term for all $k\ge 2$.

Similarly to the previous subsection, we fix an interaction range $k$ and at the moment only consider all terms that contribute to $\beta_{\ell}$, thus we get
\[
\begin{aligned}
T:=&-\frac{\phi_{xx}(k)}{2}\sum_{\ell=-M+1}^{M}\left(
k^2\frac{u_{\ell+1}-2u_{\ell}+u_{\ell-1}}{a}
\right)(u_{\ell})\\
&\quad -\frac{\phi_{xx}(k)}{2}\sum_{\ell=-M+1}^{M}\left[\bigg(\frac{u_{\ell+k}-2u_{\ell}+u_{\ell-k}}{a}-
k^2\frac{u_{\ell+1}-2u_{\ell}+u_{\ell-1}}{a}\bigg)\big(\beta_{\ell}u_{\ell}\big)
\right]\\
=&T_{1}+T_{2}.
\end{aligned}
\]
For $T_1$, we similarly have
\begin{equation}\label{T1_genN}
\begin{aligned}
T_1=&-\frac{\phi_{xx}(k)}{2}\sum_{\ell=-M+1}^{M}\left(
k^2\frac{u_{\ell+1}-2u_{\ell}+u_{\ell-1}}{a}
\right)(u_{\ell})\\
=&-\frac{\phi_{xx}(k)}{2}\sum_{\ell=-M+1}^{M}k^2\left(u'_{\ell}-u'_{\ell-1}\right) u_{\ell}= \frac{k^2}{2}\phi_{xx}(k)\|u'\|^2_{\ell_2}.
\end{aligned}
\end{equation}
For $T_2$, the simplification is more difficult, we have
\begin{equation*}
\begin{aligned}
T_{2}=&-\frac{\phi_{xx}(k)}{2}\sum_{\ell=-M+1}^{M}\left[\bigg(\frac{u_{\ell+k}-2u_{\ell}+u_{\ell-k}}{a}-
k^2\frac{u_{\ell+1}-2u_{\ell}+u_{\ell-1}}{a}\bigg)\big(\beta_{\ell}u_{\ell}\big)
\right]\\
=& -\frac{\phi_{xx}(k)}{2}\sum_{\ell=-M+1}^{M}\left[\bigg( \sum_{j=1}^{k}\sum_{s=1}^{k}u''_{\ell-j+s}\bigg)-
k^2u''_{\ell}\right]\big(a\beta_{\ell}u_{\ell}\big)\\
=& -\frac{\phi_{xx}(k)}{2}\sum_{\ell=-M+1}^{M}\left[ \sum_{j=1}^{k}\sum_{s=1}^{k}\bigg(u''_{\ell-j+s}-
u''_{\ell}\bigg)\right]\big(a\beta_{\ell}u_{\ell}\big).
\end{aligned}
\end{equation*}
Due to the exact symmetry of $j$ and $s$, we have
\[
\sum_{j=1}^{k}\sum_{s=1}^{k}u''_{\ell-j+s}
=\sum_{j=1}^{k}\sum_{s=1}^{k}u''_{\ell-s+j}.
\]
Hence, $T_2$ can be converted into a symmetrical form
\begin{equation}\label{T2_genN}
\begin{aligned}
T_2
=& -\frac{\phi_{xx}(k)}{2}\sum_{\ell=-M+1}^{M}\left[ \sum_{j=1}^{k}\sum_{s=1}^{k}\bigg(u''_{\ell-j+s}-
u''_{\ell}\bigg)\right]\big(a\beta_{\ell}u_{\ell}\big)\\
=&-\frac{\phi_{xx}(k)}{4}\sum_{\ell=-M+1}^{M}\left[ \sum_{j=1}^{k}\sum_{s=1}^{k}\bigg(u''_{\ell-j+s}-
2u''_{\ell}+u''_{\ell+j-s}\bigg)\right]\big(a\beta_{\ell}u_{\ell}\big)\\
=&-\frac{\phi_{xx}(k)}{4}\sum_{\ell=-M+1}^{M}\left[ \sum_{j=1}^{k}\sum_{s=1}^{k}\big(u''_{\ell-j+s}
-u''_{\ell}\big)-\big(u''_{\ell}-u''_{\ell+j-s}\big)\right]\big(a\beta_{\ell}u_{\ell}\big).
\end{aligned}
\end{equation}
Therefore by the discrete summation by parts formula, we have
\begin{equation}\label{T2_genN_eq2}
\begin{aligned}
T_2=&\frac{\phi_{xx}(k)}{4}\sum_{\ell=-M+1}^{M}\left[ \sum_{j=1}^{k}\sum_{s=1}^{k}\big(u''_{\ell-j+s}
-u''_{\ell}\big)\right]a\big(\beta_{\ell-j+s}u_{\ell-j+s}-\beta_{\ell}u_{\ell}\big)\\
=&\frac{-\phi_{xx}(k)}{4}\sum_{\ell=-M+1}^{M}\left[ \sum_{j=1}^{k}\sum_{s=1}^{k} u''_{\ell}a\right]\bigg(\big(\beta_{\ell-j+s}u_{\ell-j+s}-\beta_{\ell}u_{\ell}\big)-\big(\beta_{\ell}u_{\ell}-\beta_{\ell+j-s}u_{\ell+j-s}\big)\bigg)\\
=&\frac{-\phi_{xx}(k)}{4}\sum_{\ell=-M+1}^{M}\left[ \sum_{j=1}^{k}\sum_{s=1}^{k} u''_{\ell}a\right]\bigg( \beta_{\ell-j+s}u_{\ell-j+s}-2\beta_{\ell}u_{\ell}+\beta_{\ell+j-s}u_{\ell+j-s}\bigg).
\end{aligned}
\end{equation}
We can carefully work on the symmetrical term
\[
\begin{aligned}
\bigg( \beta_{\ell-j+s}u_{\ell-j+s}-&2\beta_{\ell}u_{\ell}+\beta_{\ell+j-s}u_{\ell+j-s}\bigg)\stackrel{r:=j-s}{=} \bigg( \beta_{\ell-r}u_{\ell-r}-2\beta_{\ell}u_{\ell}+\beta_{\ell+r}u_{\ell+r}\bigg)\\
&=\beta_{\ell-r}u_{\ell-r}-\beta_{\ell}u_{\ell+1}+\big(\beta_{\ell}u_{\ell+1}-2\beta_{\ell}u_{\ell}+\beta_{\ell}u_{\ell-1}\big)-\beta_{\ell}u_{\ell-1}+\beta_{\ell+r}u_{\ell+r}\\
&= \beta_{\ell}u''_{\ell}a^2+\beta_{\ell-r}u_{\ell-r}-\beta_{\ell}u_{\ell+1}+\beta_{\ell+r}u_{\ell+r}-\beta_{\ell}u_{\ell-1}.
\end{aligned}
\]
Without loss of generality, we assume $r>0$, then
\begin{equation}\label{beta_genN_eq1}
\begin{aligned}
\bigg( &\beta_{\ell-r}u_{\ell-r}-2\beta_{\ell}u_{\ell}+\beta_{\ell+r}u_{\ell+r}\bigg)
=\beta_{\ell}u''_{\ell}a^2+\beta_{\ell-r}u_{\ell-r}-\beta_{\ell}u_{\ell+1}+\beta_{\ell+r}u_{\ell+r}-\beta_{\ell}u_{\ell-1}\\
&=\beta_{\ell}u''_{\ell}a^2 +\big(\beta_{\ell}u_{\ell-r}-\beta_{\ell}u_{\ell-1}\big)+\big(\beta_{\ell}u_{\ell+r}-\beta_{\ell}u_{\ell+1}\big)+u_{\ell+r}\sum_{t=0}^{r-1}\beta'_{\ell+t}a-u_{\ell-r}\sum_{t=0}^{r-1}\beta'_{\ell-r+t}a.
\end{aligned}
\end{equation}
Therefore, we can handle $T_2$ for general $k$-th-neighbor-interaction range in a similar way to that for the case of $k=2$, and obtain that
\[
|T_{2}|\le (-\phi_{xx}(k)) \big(|\widetilde{c}_4| L^{-\frac{5}{2}}a^{-\frac{1}{2}}+|\widetilde{c}_5|L^{-2}\big)\|u'\|_{\ell^2}^2
\]
which suggests that for $k\ge 2$
\[
T_{2}\ge (\phi_{xx}(k)) \big(\widetilde{c}_4 L^{-\frac{5}{2}}a^{-\frac{1}{2}}+\widetilde{c}_5L^{-2}\big)\|u'\|_{\ell^2}^2.
\]
Combining with $T_1$ estimate \eqref{T1_genN}, we have for any $k\ge 2$
\begin{equation}\label{bqcf_bilinear_genN_eq2}
\begin{aligned}
\< F^{bqcf, lin}_{,k}&(u), u\>
\ge \phi_{xx}(k)\left(k^2+\widehat{C}_{k} L^{-\frac{5}{2}}a^{-\frac{1}{2}}+\widehat{C}_{k}L^{-2}\right)\|u'\|_{\ell^2}^2
\end{aligned}
\end{equation}
with $L=\widetilde{C}a^{-\frac{1}{5}}$ when $a=\frac{1}{M}$ being sufficiently small.

Thus, collecting all interactions up to the $N$-the Nearest Neighbor Interaction Range, we have
\[
\begin{split}
   \<F_{,}^{bqcf,lin}(u), u\> =& \<F_{\elaine{,}1}^{bqcf,lin}(u), u\>+\sum_{k=2}^{N} \<F_{\elaine{,}k}^{bqcf,lin}(u), u\>\\
  \ge& \left[\phi_{xx}(1)+\sum_{k=2}^{N} \phi_{xx}(k)\big(k^2+\widehat{C}_{k}L^{-\frac{5}{2}}a^{-\frac{1}{2}}\big)\right]\|u'\|_{\ell^2}^2.
   \end{split}
\]
\helen{Hence, similar to Theorem~\ref{theorem:blendingregion}, we summarize the stability conditions on the blending size for a general atomistic chain with $N$-th Nearest Neighbour interaction in the following theorem.}
\begin{theorem}\label{corollary: blendsize}
Suppose that the number of atoms $M$ is sufficiently large, which is equivalent to $a=\frac{1}{M}$ being sufficiently small, and the blending function $\beta$ is sufficiently smooth. Also, we assume that the fully atomistic model is stable so that \[\big[\phi_{xx}(1)+\sum_{k=1}^{N}k^2\phi_{xx}(k)\big]>0.\] If the blending size $L$ satisfies \elaine{$L=\widetilde{C} a^{-\frac{1}{5}}$}, then the linear B-QCF operator $F_{,}^{bqcf,lin}$ is positive-definite in terms of the $H^{1}$ semi-norm
\begin{equation}
    \elaine{
   \helen{\sum_{k=1}^{N} \sum_{\ell=-M+1}^{M}
    \<F^{bqcf,lin}_{\ell,k}}u,u\>  } \geq 
    \elaine{\widetilde{C}}\norm{u^{'}}_{\ell_{2}}^{2},
\end{equation}
{where $\widetilde{C}$} is strictly bounded above zero as $a=\frac{1}{M}\rightarrow0$.
\end{theorem}
Meanwhile, we can see that the bounds are dependent on the smoothness of $\beta$. Therefore, we aim to find the optimal types of blending function to preserve the bounds in the coming numerical session.
\begin{remark}Throughout, the term `optimal' is used to describe the conclusion that $L=\widetilde{C}a^{-\frac{1}{5}}$ is the optimal blending size for keeping coercivity of the force operator when $a=\frac{1}{M}\rightarrow 0$. Optimality in these instances describes the smallest asymptotic order that could be expected to ensure stability. However, the analysis of consistency conditions is beyond the purview of this paper.  Note that $\widehat{C}$ is dependent on the blending function, $\beta$, and the choice for potential energy, $\phi$; and it is not on the lattice spacing constant, $a$.

Also notice from the inequality \eqref{bqcf_bilinear_genN_eq2} that
$\widehat{C}_k L^{-\frac{5}{2}}a^{-\frac{1}{2}}$ dominates the other terms only when the lattice spacing $a=\frac{1}{M}$ is very small. As a result, the asymptotically rate $L\sim a^{-\frac{1}{5}}= M^{\frac{1}{5}}$ might be not observed when the size of atomistic chain is only moderately large. In this case, we suggest to take $L\sim a^{-\frac{1}{3}}=M^{\frac{1}{3}}$, which is observed in the numerical test.
\end{remark}

\section{Numerical Experiments}
We conduct numerical experiments to verify the theoretical findings from the stability analysis.

\subsection{Blending Size and the Stability Constant}
We consider a periodic chain with atom indices from $-M+1$ to $M$. For the following numerical simulations, we set $M=2000$. The optimal blending size as was found analytically in the previous section is $M^{1/5} \approx 5$.  We will test to see if the numerical experiments coincide with this value.

First, we apply the uniform stretch to the atomistic chain and compute the critical strain when $F^{a, lin}$ and $F^{bqcf, lin}$ loses the coercivity. We compare the critical strain values between the atomistic model and blending models with different blending sizes and various types of blending functions. By comparing these values to the atomistic critical strain error, we obtain the optimal blending function and attempt to verify the optimal blending size of $M^{\frac{1}{5}}$ found before.

In numerical experiments, the Morse Potential,
\begin{equation*}
    \phi(r)=D_e\times [1-e^{-\alpha(r-r_e)}]^2,
\end{equation*}
was utilized for the interaction potential due to its popularity in applications. We use the values $D_e=3$ 
and $\alpha=3, 4,$and $ 5$, respectively. Recall from Figure \ref{fig:1D_morsePot}, the local minimal value is set to be $\phi(1)$ and the local height of the potential is $D_e$. Also, as $\alpha$ grows larger, the more narrow the potential becomes.

We will consider our computational domain as $\Omega=(-1,1]$ with periodic boundary conditions. 
The computational domain will be decomposed into several sub-domains following \cite{seleson2013aa}. An interaction range is introduced to serve as a buffer region to simplify the treatment of periodic boundary conditions. 
The lattice spacing constant that was used had a value of $a=\frac{1}{M}$ which helped to start the blending region.

For the numerical experiments, we denote the blending region, $\Omega^{b}=(b_{1},b_{2})$ for some $b_{1}, b_{2} \in \Omega$. The numerical blending size, $L$, will be defined as $L=b_{2}-b_{1}$. We can compare the numerical blending size to that which was found in the stability analysis.

Recall from Definition \eqref{blendingfunction},
\begin{equation*}\label{blending}
    \beta(x)=\begin{cases}
    1, \quad &x \in \Omega^{a} \\
    0, \quad &x \in \Omega^{c} \\
    \in(0,1), \quad &x \in \Omega^{b}=(b_{1},b_{2}).
    \end{cases}
\end{equation*}

We conduct numerical experiments using a piecewise linear spline, piecewise cubic spline, and piecewise quintic spline blending function which are defined as follows:
\begin{equation*}\label{def_linear}
    \beta^{linear}(x):=\begin{cases}
    1, \quad &x \in \Omega^{a}, \\
    0, \quad &x \in \Omega^{c}, \\
    1-\frac{x-b_{1}}{L}, &x \in \Omega^{b},
    \end{cases}
\end{equation*}
and
\begin{equation*}\label{def_cubic}
    \beta^{cubic}(x):=\begin{cases}
    1, \quad &x \in \Omega^{a}, \\
    0, \quad &x \in \Omega^{c}, \\
    1+2(\frac{x-b_{1}}{L})^{3}-3(\frac{x-b_{1}}{L})^{2}, &x \in \Omega^{b},
    \end{cases}
\end{equation*}
and
\begin{equation*}\label{def_quintic}
    \beta^{quintic}(x):=\begin{cases}
    1, \quad &x \in \Omega^{a}, \\
    0, \quad &x \in \Omega^{c}, \\
    1-6(\frac{x-b_{1}}{L})^{5}+15(\frac{x-b_{1}}{L})^{4}-10(\frac{x-b_{1}}{L})^{3} &x \in \Omega^{b}.
    \end{cases}
\end{equation*}
Graphical demonstrations of these various blending functions can be found in Figure~\ref{fig:1D_blendFnc}.

\begin{figure}[htp!]
    \centering
    \includegraphics[width=0.47\textwidth]{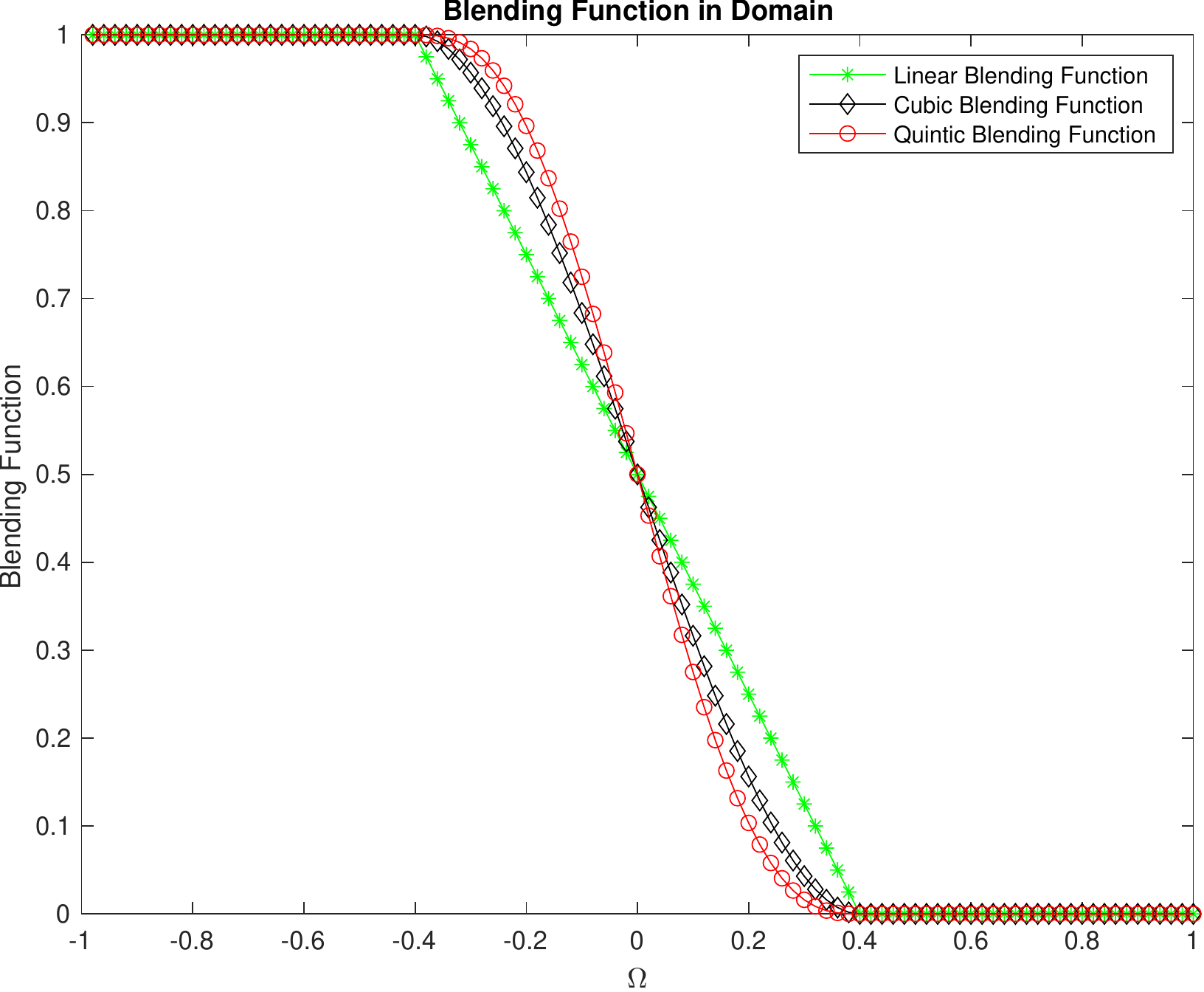}
    \caption{Pictoral representation of the blending functions used in numerical experiments. Recall, at $\beta=1$, the purely atomistic model is obtained and at $\beta=0$, the purely continuum model is obtained.}
    \label{fig:1D_blendFnc}
\end{figure}

 We apply a uniform stretch to the atomistic chain. From this, we numerically compute the critical strains of the atomistic model and compare this to the coupling model with different blending sizes in pursuit of the critical stretch value that makes the atomistic chain unstable. The step size for increasing $\gamma$ is $\Delta \gamma=10^{-5}$. We also model the different values of $\alpha$ in the Morse Potential using the cubic blending function. As can be seen in Table 1, the cubic blending function reaches the atomistic critical stretch value quicker than the other two blending functions.

 \begin{figure}[htp!]
    \centering
    \subfigure{\includegraphics[ width=0.45\textwidth]{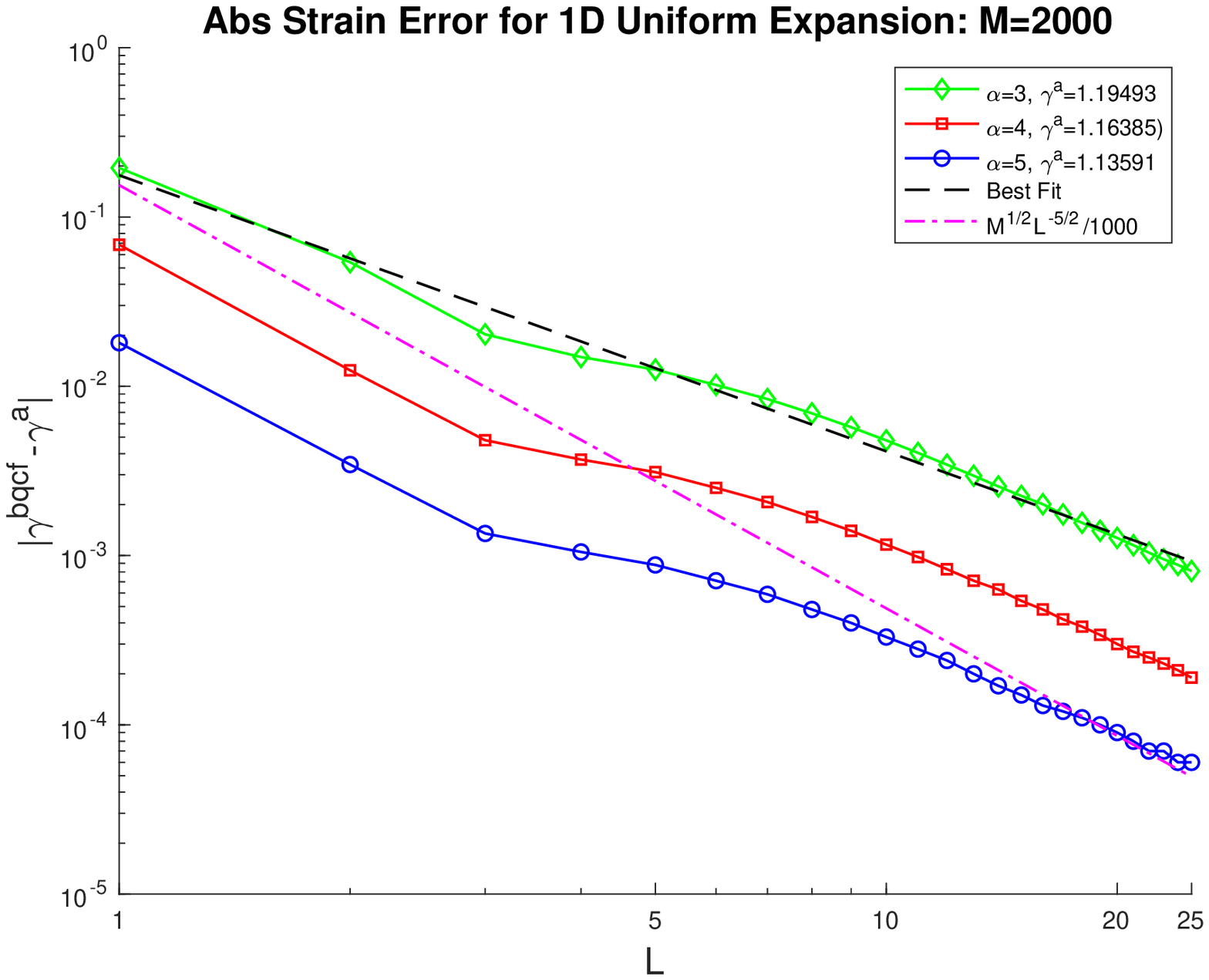}}
    \subfigure{\includegraphics[width=0.45\textwidth]{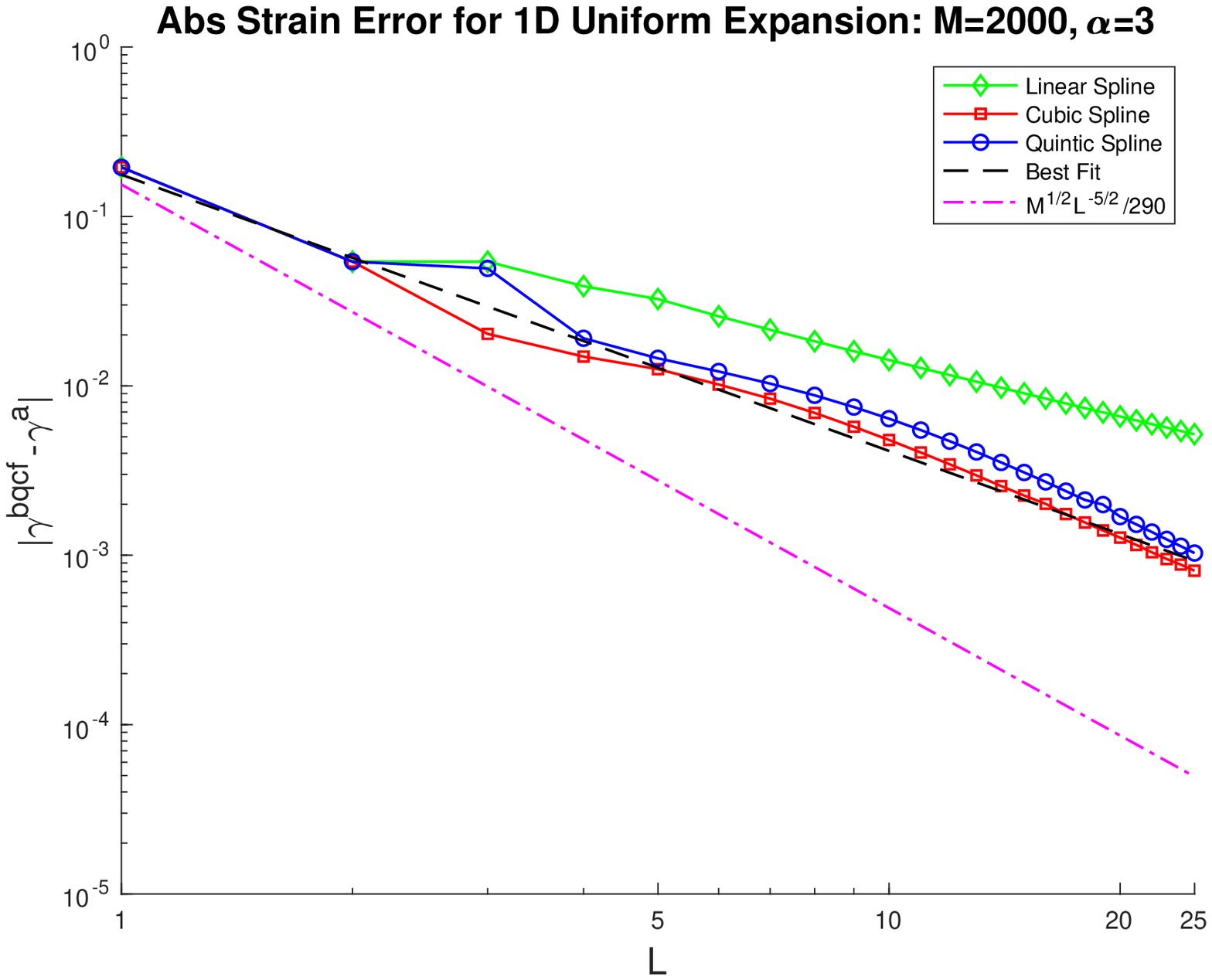}}
    \caption{The absolute critical strain errors are plotted for the 1D uniform stretching. We set $M$=2,000 and $\gamma^{a}$ and $\gamma^{bqcf}$ are the critical strains for the atomistic and B-QCF models, respectively. a) Models the cubic spline blending for various values of $\alpha$; and b) models the critical strain errors of linear, cubic, and quintic blending functions with $M=2,000$ and $\alpha=3$. }
    \label{fig:Error}
\end{figure}

\begin{table}[htp!]
            \centering
            \begin{tabular}{c|c|c|c}
 Blend size  & Linear & Cubic & Quintic  \\
 \hline
                 1 & 1&1 &1\\
 2 & 1.1400 & 1.1409 & 1.1409\\
 3 & 1.1269 & 1.1747 & 1.1456\\
 4 & 1.1562 & 1.1801 & 1.1759\\
 5 & 1.1624 & 1.1824 & 1.1804\\
 6 & 1.1692 & 1.1848 & 1.1828\\
 7 & 1.1735 & 1.1866 & 1.1847\\
 10 &1.1811 &1.1950 & 1.1950 \\
 \hline
            \end{tabular}
            \caption{Shown are the critical stretching values for linear, cubic, and quintic blending models for a blending size from $a$ to $10a$. The critical value for the purely atomistic model was found to be $\gamma^a=1.195$. The numerical increment for $\Delta \gamma$ is $10^{-5}$. }
            \label{tab:my_label}
        \end{table}

The results from Table 1 suggest the blending size to be $L \approx M^{\frac{1}{3}}$, this might be due to the other terms in the inequality \eqref{bqcf_bilinear_genN_eq2} when observing only a moderately large atomistic chain. \elaine{For further clarification, see Remark \ref{remark:inequalities}.} 

The results from the numerical experiments find the cubic blending function as that which converges quickest toward the atomistic strain value and is thus the optimal blending function from those we tested. 
\elaine{
\begin{remark}\label{remark:inequalities}
Recall from \eqref{inequalities}, the assumption that:
\begin{equation*}
    \left(2c_{3}L^{-\frac{5}{2}}a^{-\frac{1}{2}}+4c_{2}L^{-2}+c_{1}L^{-1}\right)\leq \left(c_{4}L^{-\frac{5}{2}}a^{-\frac{1}{2}}\right).
\end{equation*}
The latter two terms on the left side of the inequality would not necessarily be negligible if the number of atoms $M$ were not large enough. This accounts for the difference observed in the blending size between the analysis and the numerical simulation.

It must be noted that the analysis conducted in the previous section only applies to the cubic or quintic blending function used in these simulations. Due to less regularities near the boundaries of the blending region, the analysis does not encompass the linear blending function. In these experiments, we can also see that the linear blending leads to the most discrepancies. 

\end{remark} }

\subsection{Simulation of Deformed Configuration}
Now that we found that the cubic blending function as the optimal blending function, we utilize this for the remaining numerical tests. We use a blend size \elaine{$L=5$ since $2000^{\frac{1}{5}}\approx4.57$} as well as $\alpha=3$ for both numerical experiments. Next, we test two functions with periodic boundary conditions as the external force of the system to ensure the blended coupling scheme performs as imagined.
\begin{itemize}
\item {\bf First, we use a sinusoidal external force}
\begin{equation*}
    F^{ext}_{\ell\elaine{,}}=0.01 a \times \sin(-x_{\ell} \times \pi).
\end{equation*}
We use $a=\frac{1}{5}$ numerically and incorporate $-1$ into sine because of our left domain boundary.

 We obtain the expected force plot for our domain as can be seen in \ref{fig:SineExtForce}. The displacement is also observed for an interaction range potential from first neighbor up to third neighbor. We observe that the difference between an interaction range potential of $N=2$ versus an interaction range potential of $N=3$ is much smaller than the difference between the change in displacement for the interaction range potential for $N=1$ and $N=2$. \elaine{Recall, the features of the Morse Potential are such that $\phi_{xx}(1)>0$ and $\phi_{xx}(k)\leq0$ for $k\geq2$. Thus, after the next-nearest neighbor, $k=2$, the change in displacement will not differ by as much.}

\begin{figure}[htp!]
    \centering
    \mbox{\subfigure{\includegraphics[width=0.45\textwidth]{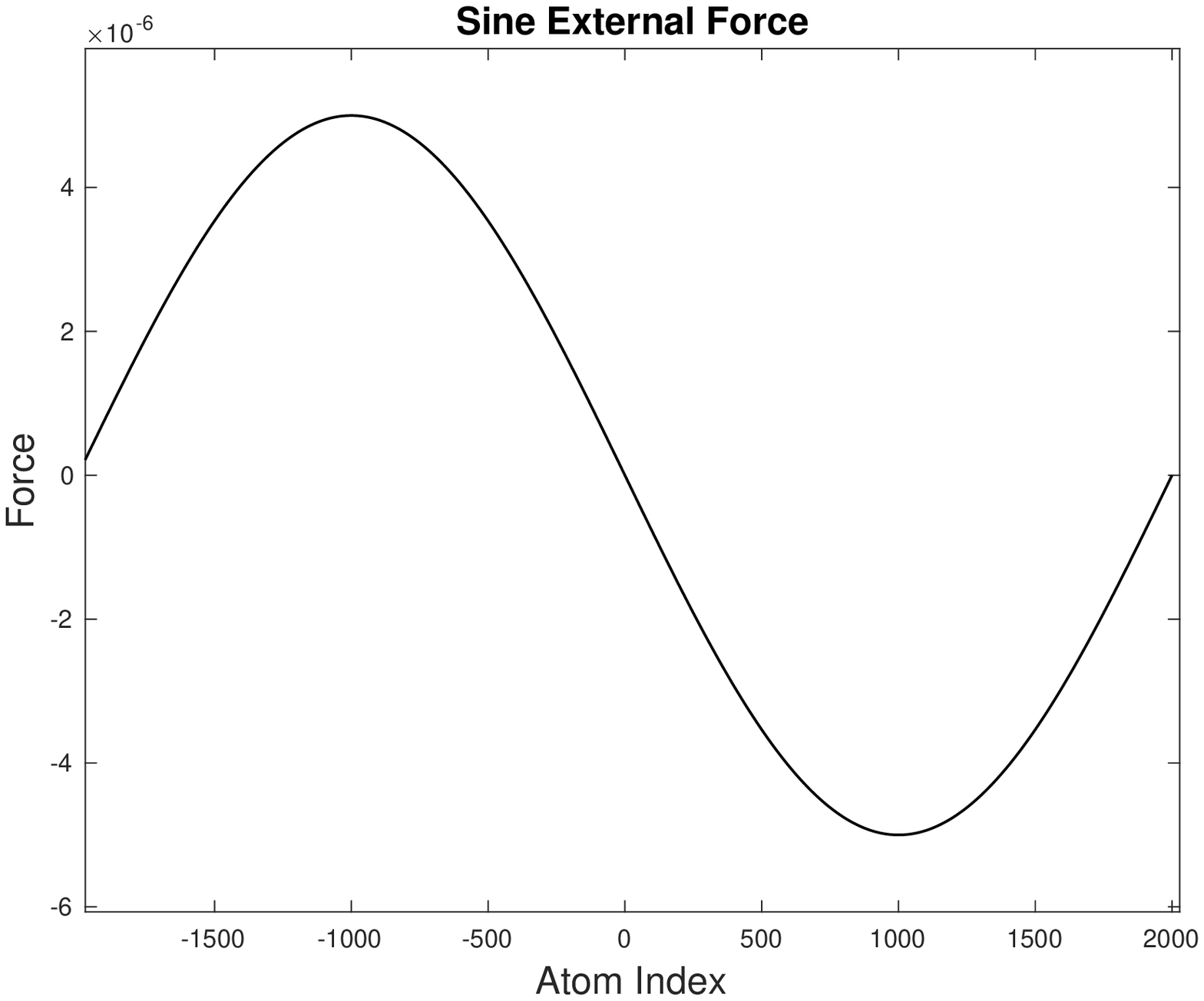}}\
    \subfigure{\includegraphics[width=0.45\textwidth]{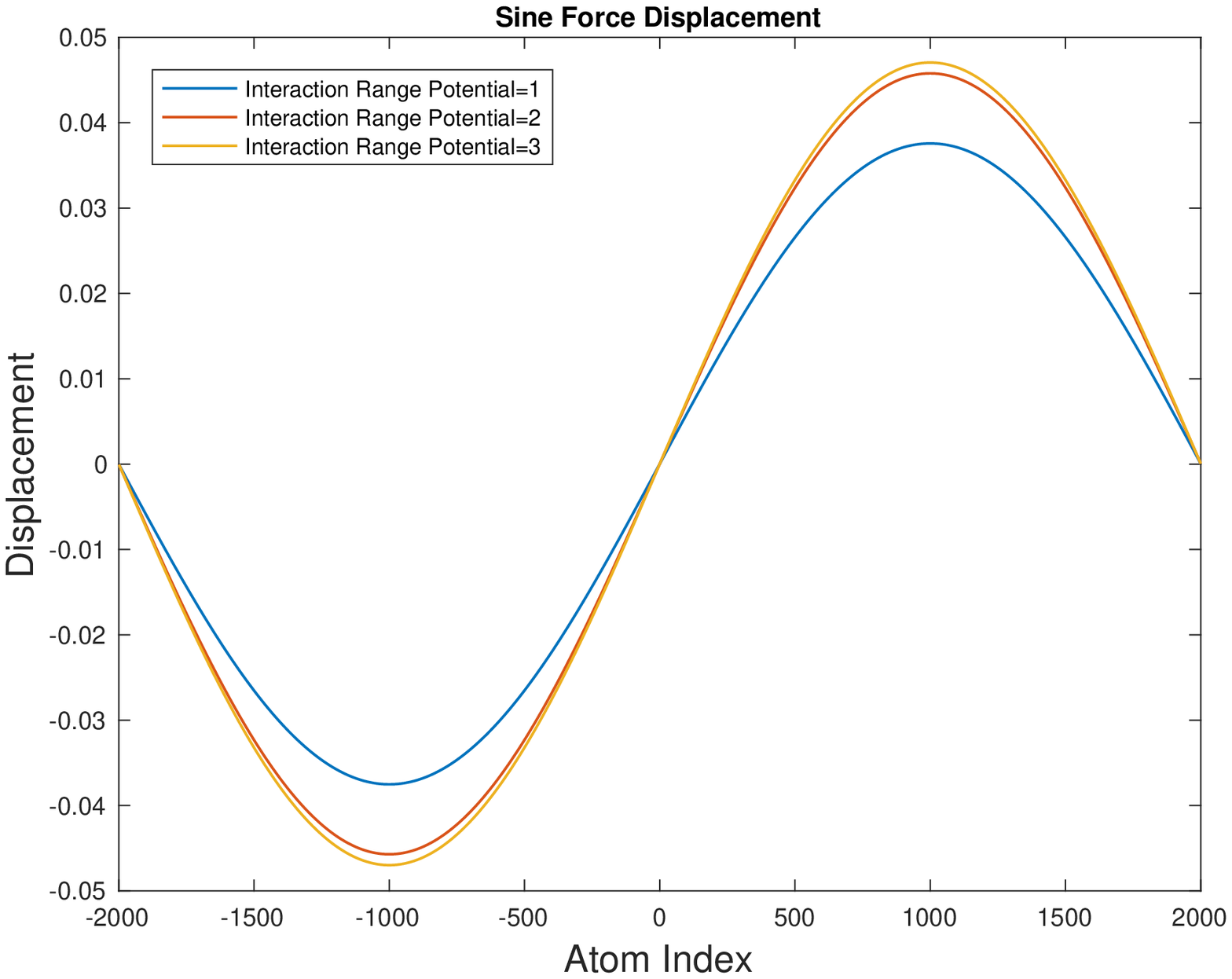}}}
    \caption{a) A sinusoidal external force is shown. b) The various displacements for this external force are displayed within the domain.}
    \label{fig:SineExtForce}
\end{figure}

\item{ \bf Next, we test a Gaussian external force:}
\begin{equation*}
    F^{ext}_{\ell\elaine{,}}=0.01\times a \times e^{\frac{-(x_{\ell}-\mu)^2}{2 \sigma^2}},
\end{equation*}
where $\mu=4a$, $\sigma=50a$, and $a=1/M$ with $M=2000$ was used.

Again, we show the force output for our domain and show the various displacements for three interaction range potentials. Similarly to the sinusoidal external force, once the interaction range reaches a value of $N=2$, the change in displacement becomes less significant.

\begin{figure}[htp!]
    \centering
    \mbox{\subfigure{\includegraphics[width=0.45\textwidth]{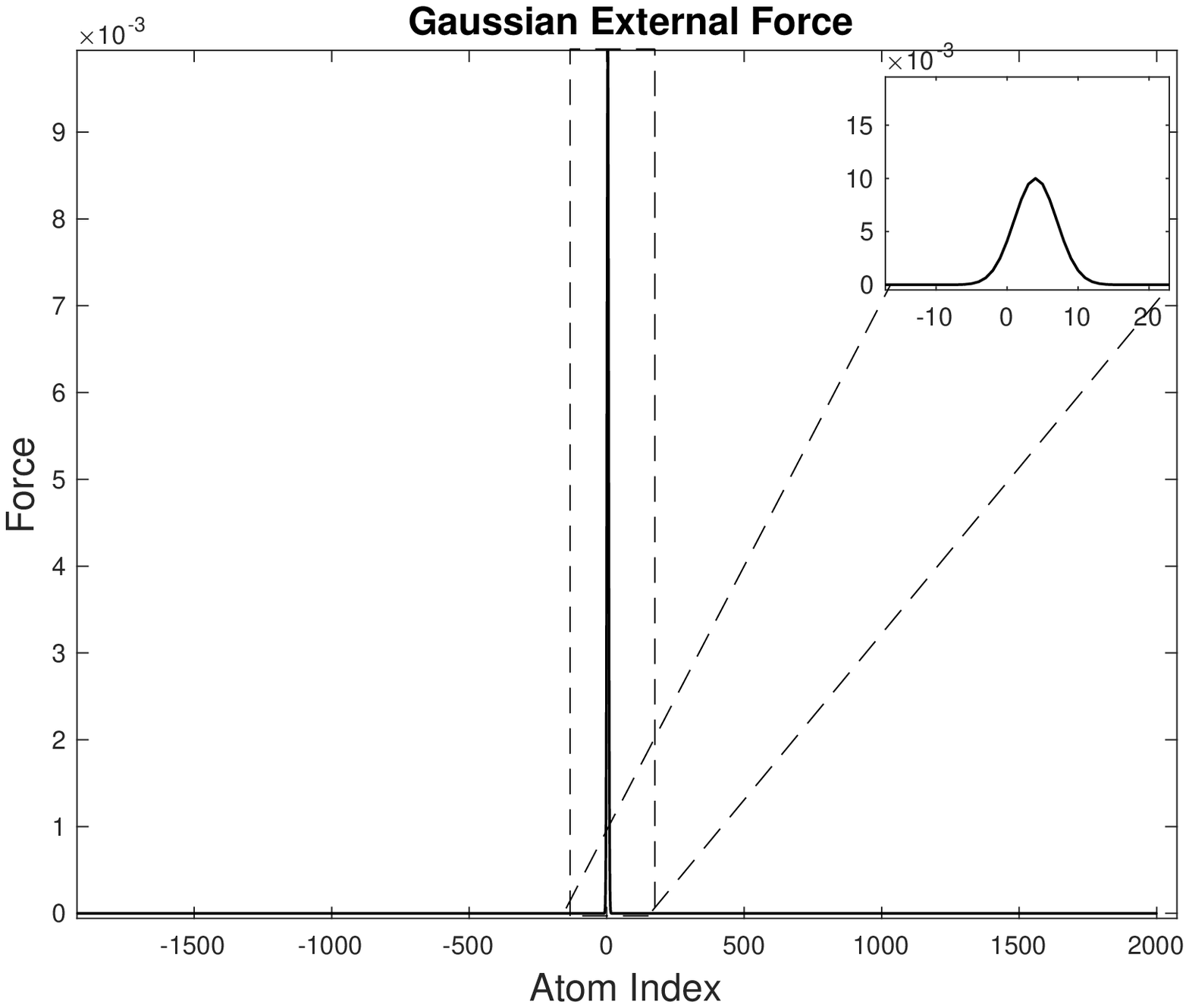}}\
    \subfigure{\includegraphics[width=0.45\textwidth]{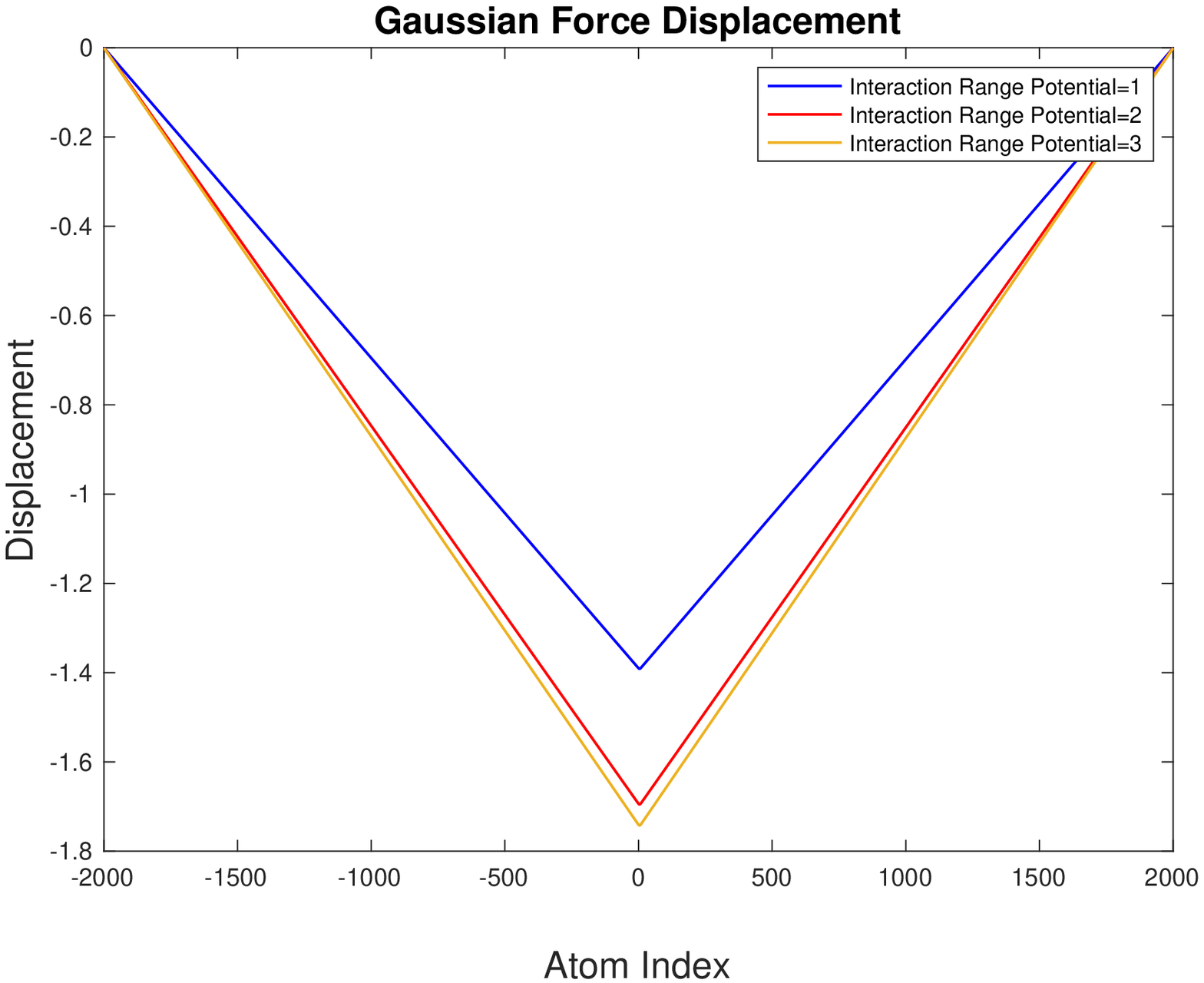}}}
    \caption{a) A Gaussian external force is shown. b) The various displacement for this external force are displayed within the domain.}
    \label{fig:GaussianExtForce}
\end{figure}

\end{itemize}
\section{Conclusion}
In this paper, inspired by the force-based coupling of the peridynamics model of \cite{seleson2013aa}, we have formulated a similar symmetric and consistent blended force-based atomistic-to-continuum coupling scheme in one-dimensional space. We were able to identify the optimal asymptotic conditions on the width of the blending region,  $L \approx M^{\frac{1}{5}}$ to ensure the $H^1$ stability of the linearized force-blending operator when chain size is huge.

We have verified the theoretical findings with numerical experiments on the blending function and the blending region. From these numerical experiments, we find that the cubic blending provides the best results compared to the critical stretch of the fully atomistic model. We also find that the optimal blending width from these numerical experiments is $L \approx M^{\frac{1}{3}}$ due to non-negligible terms when {\it not working with} a large enough atomistic chain.

In the future, extension of this scheme to two-dimensional atomistic-to-continuum coupling with a triangular crystal lattice in regards to the neighbors will be pursued.

\section{Acknowledgements}
Elaine Gorom-Alexander and Dr. X. Li are supported by  NSF CAREER award: DMS-1847770 and the University of North Carolina at Charlotte Faculty Research Grant.

We  would like to thank the helpful discussions from Dr. Pablo Seleson and Dr. Christoph Ortner. We also want to thank the opportunity provided by the 50th John H. Barrett Memorial Lectures organizing committee.

\section{Appendix}\label{appendix}
\subsection{\elaine{A More Rigorous Proof for Proposition \ref{proposition:forceconsistency}, A Consistency Analysis of Force}}\label{RigorProof}

\begin{proof}
Comparing $F^{c,lin}$ and $F^{a, lin}$, we have

\begin{equation*}
        F_{\ell,}^{c,lin}(u)-F_{\ell,}^{a,lin}(u)=-\sum_{\substack{k=-N,\\k\neq{0}}}^{N}\frac{1}{2}\phi_{xx}(k)\left(k^{2}u_{\ell}^{''}\right)+\sum_{\substack{k=-N, \\k \neq{0}}}^{N}\frac{1}{2}\phi_{xx}(k)\left(\frac{u_{\ell+k}-2u_{\ell}+u_{\ell-k}}{a^2}\right).
\end{equation*}
Let $N \in \mathbb{N}$ be fixed. For any $k=-N,\dots,N$ and $k\neq 0$, we apply the Taylor expansion to $\widetilde{u}_{\ell+k}={u}_{\ell+k}$ and $\widetilde{u}_{\ell-k}={u}_{\ell-k}$ around $\ell$. \elaine{We compare the differencing notation used to define the discrete displacement field with choosing a smooth spline interpolation.} We proceed with defining $\widetilde{u}_{x}$ as this smooth interpolation of the discrete displacement field $u$ in order to compute this approximation.
\begin{equation*}
    \begin{split}
        u_{\ell+k}&= u_{\ell}+ka\widetilde{u}_{x}+\frac{1}{2}(ka)^2\widetilde{u}_{xx}+\frac{1}{6}(ka)^3\widetilde{u}_{xxx}+O(a^4), \\
        u_{\ell-k}&= u_{\ell}-ka\widetilde{u}_{x}+\frac{1}{2}(ka)^2\widetilde{u}_{xx}-\frac{1}{6}(ka)^3\widetilde{u}_{xxx}+O(a^4).
    \end{split}
\end{equation*}

Thus, for all $k$,
\begin{equation}\label{Expand}
    \begin{split}
        u_{\ell+k}-2u_{\ell}+u_{\ell-k}&=u_{\ell}+ka\widetilde{u}_{x}+\frac{1}{2}(ka)^2\widetilde{u}_{xx}+\frac{1}{6}(ka)^3\widetilde{u}_{xxx}+O(a^4)-2u_{\ell} \\
        &+u_{\ell}-ka\widetilde{u}_{x}+\frac{1}{2}(ka)^2\widetilde{u}_{xx}-\frac{1}{6}(ka)^3\widetilde{u}_{xxx}+O(a^4) \\
        &=(ka)^{2}\widetilde{u}_{xx}+O(a^{4}).
    \end{split}
\end{equation}

\noindent Utilizing \eqref{Expand} for the atomistic and continuous force equations, the consistency analysis yields:
\begin{equation*}
    \begin{split}
        F_{\ell,}^{c,lin}&(u)-F_{\ell,}^{a,lin}(u)\\&=
        -\sum_{\substack{k=-N,\\k\neq{0}}}^{N}\frac{1}{2}\phi_{xx}(k)\left(k^{2}u_{\ell}^{''}\right)+\sum_{\substack{k=-N, \\k \neq{0}}}^{N}\frac{1}{2}\phi_{xx}(k)\left(\frac{u_{\ell+k}-2u_{\ell}+u_{\ell-k}}{a^2}\right)
        \\
        &=-\sum_{\substack{k=-N,\\k\neq{0}}}^{N}\frac{1}{2}k^{2}\phi_{xx}(k)\left(\frac{u_{\ell+1}-2u_{\ell}+u_{\ell-1}}{a^{2}}\right)
        +\sum_{\substack{k=-N, \\k \neq{0}}}^{N}\frac{1}{2}\phi_{xx}(k)\left(\frac{u_{\ell+k}-2u_{\ell}+u_{\ell-k}}{a^2}\right) \\
        &= -\sum_{\substack{k=-N,\\k\neq{0}}}^{N}\frac{1}{2}k^{2}\phi_{xx}(k)\frac{a^{2}\widetilde{u}_{xx}+O(a^{4})}{a^{2}}+\sum_{\substack{k=-N, \\k \neq{0}}}^{N}\frac{1}{2}k^{2}\phi_{xx}(k)\frac{(ka)^{2}\widetilde{u}_{xx}+O(a^{4})}{a^{2}}\\
        &=-\sum_{\substack{k=-N,\\ k\neq{0}}}^{N}\frac{1}{2} k^{2}\phi_{xx}(k)\widetilde{u}_{xx}+\sum_{\substack{k=-N, \\k \neq{0}}}^{N}\frac{1}{2}k^{2}\phi_{xx}(k)\widetilde{u}_{xx}+O(a^2)\\
    &= O(a^2).
    \end{split}
\end{equation*}
Thus, the consistency error between the linearized atomistic force equation and the continuum force equation is $O(a^{2})$.
\end{proof}

\subsection{Analysis in the continuous setting}
Before finding the nearest neighbor and the next-nearest neighbor interaction for the discrete case, the continuous case was observed. The continuous case was meant to shed light on the nature of the discrete case as it would be easier to find.

From \eqref{def_linear_bqcf} we look at the next-nearest neighbor interaction. Also, we will approximate
$\frac{\beta_{\ell+k}+2\beta_{\ell}+\beta_{\ell-k}}{4}\approx\beta(x_{\ell})$.
Thus, \eqref{def_linear_bqcf} becomes for the force-based operator:
\begin{equation*}
\begin{split}
    F_{\elaine{\ell,}}^{bqcf,lin}
    &=\left(\frac{\beta_{\ell-1}+2\beta_{\ell}+\beta_{\ell+1}}{4}\right)\phi_{xx}(1)\left(\frac{u_{\ell+1}-2u_{\ell}+u_{\ell-1}}{a^{2}}\right)+\left(1-\frac{\beta_{\ell-1}+2\beta_{\ell}+\beta_{\ell+1}}{4}\right)\phi_{xx}(1)u_{\ell}^{''} \\
    &+\left(\frac{\beta_{\ell-2}+2\beta_{\ell}+\beta_{\ell+2}}{4}\right)\phi_{xx}(2)\left(\frac{u_{\ell+2}-2u_{\ell}+u_{\ell-2}}{a^{2}}\right)+\left(1-\frac{\beta_{\ell-2}+2\beta_{\ell}+\beta_{\ell+2}}{4}\right)\phi_{xx}(2)4u_{\ell}^{''} \\
    &\approx \beta(x_{\ell})\left(\phi_{xx}(1)u_{\ell}^{''}+\phi_{xx}(2)\left(\frac{u_{\ell+2}-2u_{\ell}+u_{\ell-2}}{a^{2}}\right)\right)+\left(1-\beta(x_{\ell})\right)\left(\phi_{xx}(1)u_{\ell}^{''}+4\phi_{xx}(2)u_{\ell}^{''}\right).
    \end{split}
\end{equation*}

Using a Taylor approximation on $u_{\ell+2}$ and $u_{\ell-2}$, the next-nearest neighbor operator becomes
\begin{equation}
    F_{\elaine{\ell,}}^{bqcf,lin}=\phi_{xx}(1)u_{\ell}^{''}+\beta(x_{\ell})\phi_{xx}(2)\left(4u_{\ell}^{''}+\frac{4}{3}u_{\ell}^{(4)}a^{2}\right)+\left(1-\beta(x_{\ell})\right)4\phi_{xx}(2)u_{\ell}^{''}.
\end{equation}
Since the nearest neighbor interaction is not difficult to find, we drop it from the continuous case to find the approximation for the next-nearest neighbor as well as utilize the fact that $\phi$ is a Lennard Jones type potential. Also, we denote $\beta(x_{\ell})$ by $\beta$ when there is no ambiguity. Therefore, the continuous next-nearest operator becomes
\begin{equation}
\begin{split}
    Fu_{x}&=\beta(-u_{xx}+a^{2}Au_{xxxx})+(1-\beta)(-u_{xx}) \\
    &=\beta a^{2}Au_{xxxx}-u_{xx}
    \end{split}
\end{equation}
where $A=\frac{4}{3}$.

\begin{lemma} \label{Appendix_lmma}For any displacements $u=(u_{\ell})_{\ell \in \mathbb{Z}}$ from $y_{\ell}$, the nearest neighbor and the next-nearest neighbor interaction operator can be written in the form
\begin{equation}
    \begin{split}
    \label{eqn: A.1}
        &<F_{\elaine{,}1}^{bqcf,lin},u>= \norm{u_{x}}^{2} \\
        &<F_{\elaine{,}2}^{bqcf,lin},u>= 4\norm{u_{x}}^{2}+a^{2}16A\norm{\sqrt{\beta}u_{xx}}^{2}+R+S.
    \end{split}
\end{equation}
where $R$ and $S$ are given by:
\begin{equation}
    \label{eqn: A.2}
        R=-a^{2}A\int\beta_{xx}(u_{x})^{2}dx,  \quad
        S=a^{2}A\int\beta_{xx}(u_{xx}u)dx.
\end{equation}
\end{lemma}

\begin{proof} Since the proof of the first identity of Lemma ~\ref{Appendix_lmma} is straightforward and utilizes the same properties, the proof for the second-neighbor interaction operator will be given. The main tool used is integration by parts based on the periodic boundary conditions.
\begin{equation}
    \begin{split}
    \label{eqn: A.3}
    <F_{\elaine{,}2}^{bqcf}u,u>
    &= \int(\beta a^{2}Bu_{xxxx}-u_{xx})udx \\
    &=\int -u_{xx}udx+\int \beta a^2Au_{xxxx}udx \\
    &=\int(u_{x})^{2}dx+a^2A\int \beta uu_{xxxx}dx \\
    &=\norm{u_x}^2-a^2A\int (\beta u)_{x}u_{xxx}dx \\
       &=\norm{u_x}^2+a^2A\norm{\sqrt{\beta}u_{xx}}^{2}+a^{2}A\int\beta_{x}(u_{x})^{2}_{x}dx-a^2A\int\beta_{x}(u\,u_{xx})_{x}dx \\
    &=\norm{u_x}^2+a^2A\norm{\sqrt{\beta}u_{xx}}^{2}-a^{2}A\int\beta_{xx}(u_{x})^{2}dx+a^{2}A\int\beta_{xx}(u_{xx}u)dx \\
    &=\norm{u_x}^2+a^2A\norm{\sqrt{\beta}u_{xx}}^{2}+R+S.
    \end{split}
\end{equation}

Using the continuous analysis as a road-map, we thus derive the discrete analysis in Section 3.
\end{proof}

\newpage
\subsection{Symbols and Notation}
\begin{center}
    \begin{tabular}{c|c}
    \hline
    \elaine{superscript e.g., $^{'}$, $^{''}$  }   & first order backward finite difference, second order central finite difference, etc. for the discrete case  \\
    \elaine{subscript e.g., $_{x}$, $_{xx}$} & first derivative, second derivative, etc. for the continuous case \\
    $\frac{\delta}{\delta u_{\ell}}$ & first order variation evaluated at $u_{\ell}$ \\
    $\frac{d}{dx}$ & derivative evaluated at $x$ \\
    $\norm{\boldsymbol{\cdot}}_{\ell_{2}}$     & $\ell_2$-norm \\
    $\norm{\boldsymbol{\cdot}}_{\ell_{\infty}}$ & $\infty$ -norm \\
    $\norm{\boldsymbol{\cdot}}_{\ell_{2}(\Omega^{b})}$ & $2$-norm evaluated on the blending region $\Omega_b$\\
    $|\boldsymbol{\cdot}|_{H^{1}}$ & discrete $H^{1}$ semi-norm \\
    $<\boldsymbol{\cdot},\boldsymbol{\cdot}>$ & inner product \\
    $\phi$ & atomistic interaction potential per unit cell \\
    $(\boldsymbol{\cdot})^{c}$ & continuum equation \\
    $(\boldsymbol{\cdot})^{a}$ & atomistic equation \\
    $(\boldsymbol{\cdot})^{lin}$ & linearized equation \\
    $(\boldsymbol{\cdot})^{bqcf}$ & force-based blended equation \\
    $\Omega$ & whole domain \\
    $\Omega^{a}$ & atomistic domain \\
    $\Omega^{c}$ & continuum domain \\
    $\Omega^{b}$ & blending domain  \\
    $\beta$ & blending function \\
    $L$ & number of atoms within the blending region $\Omega_b$ \\
    $a$ & lattice spacing constant \\
    $y_{\ell}$ & deformation at $\ell$\\
    $u_{\ell}$ & displacement at $\ell$\\
    $E$ & energy equation\\
    $F_{\ell,k}$ & force equation at atom $\ell$ with $k^{th}$ neighbor interaction\\
    $F_{\ell,}$ & force equation at atom $\ell$ with the summation of all neighbor interaction \\
    $F_{,k}$ & force equation with $k^{th}$ neighbor interaction with the summation of all atoms \\
    $\alpha$ & parameter in the Morse Potential\\
    $N$ & number of neighbors within the interaction range\\
    $M$ & half number of atoms in the atomistic chain \\
    $F^{ext}$. & external force
    \end{tabular}
\end{center}




\bibliographystyle{plain}
\bibliography{reference}
\end{document}